\documentclass{article}
\usepackage{a4wide}
\usepackage{amsmath,amsfonts,amssymb,mathrsfs,amsthm}
\usepackage{float}
\usepackage{srcltx}
\usepackage{slashed}

\usepackage{authblk}

\title{Asymptotic behavior of small solutions to the Vlasov--Klein--Gordon system in high dimensions}

\author{Ho Lee\footnote{holee@khu.ac.kr}}

\affil{Department of Mathematics, College of Sciences, Kyung Hee University, Seoul 02447, Republic of Korea}

\newtheorem{theorem}{Theorem}[section]
\newtheorem{lemma}{Lemma}[section]

\newtheorem{proposition}{Proposition}[section]
\newtheorem{remark}{Remark}[section]

\newcommand{\bbr}{{\mathbb R}}
\newcommand{\bbs}{{\mathbb S}}

\newcommand{\bbn}{{\mathbb N}}

\newcommand{\bbp}{{\mathbb P}}

\newcommand{\bbrn}{{\bbr^n}}

\newcommand{\rV}{{\rm V}}
\newcommand{\rK}{{\rm KG}}

\newcommand{\cP}{{\mathcal P}}
\newcommand{\cQ}{{\mathcal Q}}

\def\charf {\mbox{{\text 1}\kern-.24em {\text l}}}

\begin{document}

\maketitle

\begin{abstract}
We study the asymptotic behavior of small solutions to the Vlasov--Klein--Gordon system in high dimensions. The standard argument of Glassey and Strauss \cite{GS87} for studying small solutions to the Vlasov--Maxwell system does not apply to the Vlasov--Klein--Gordon system due to the massiveness of the Klein--Gordon field. In this paper we use the vector field method and consider solutions in dimensions $ n \geq 4 $ with the hyperboloidal foliation of the Minkowski spacetime to obtain the asymptotic properties for the Vlasov--Klein--Gordon system.
\end{abstract}


%
\section{Introduction}
Let us consider an ensemble of relativistic particles interacting through a self-consistent field. In many situations this field is given by the Maxwell field, but in this paper we instead consider the Klein--Gordon field. Assuming that the particles are collisionless, we obtain the Vlasov--Klein--Gordon (VKG) system \cite{KRST03, KRST05}:
\begin{align*}
\partial_t f + \hat{ v } \cdot \nabla_x f - \nabla_ x \phi \cdot \nabla_v f & = 0 , \\
- \partial_t^2 \phi + \Delta_x \phi - \phi & = \int_{ \bbr^n } f \, dv .
\end{align*}
Here, $ f = f ( t , x , v ) $ and $ \phi = \phi ( t , x ) $ with $ t \in \bbr $, $ x \in \bbr^n $ and $ v \in \bbr^n $ are the distribution function and the Klein--Gordon (scalar) field, respectively, and $ \hat{ v } = v / v^0 $ denotes the relativistic velocity of a particle with momentum $ v $. We consider particles of unit mass so that $ v^0 $ is given by
\begin{align*}
v^0 = \sqrt{ 1 + | v |^2 } .
\end{align*}
In this paper we are interested in small solutions to the VKG system and study the asymptotic behavior of them by using the vector field method.

The VKG system was first studied in \cite{KRST03, KRST05} to obtain mathematical insight into the global existence problem for the Vlasov--Maxwell (VM) system, which still remains unsolved for large initial data. By replacing the massless Maxwell field with a massive Klein--Gordon field, one obtains a faster temporal decay rate, suggesting a more favorable scenario for global existence. However, this enhanced decay rate does not necessarily simplify the problem. Indeed, for large initial data, the global existence for the VKG system was obtained in \cite{KRST05}, but it was under the same continuation criterion as in the VM case \cite{GS86}.

For small initial data, the massiveness of the Klein--Gordon field even complicates the problem. While the Maxwell field decays at a rate of $ t^{ - 1 } $, its masslessness allows for an additional decay rate of $ ( t - r )^{ - 1 } $ inside the light cone. Then, by assuming compact support, this extra decay enables the application of the standard argument of Bardos and Degond \cite{BD85}, which leads to the global existence of small solutions to the VM system \cite{GS87}. In contrast, the Klein--Gordon field decays at a faster rate of $ t^{ - 3 / 2 } $, but this rate alone is insufficient to apply the argument of \cite{BD85}. One needs to derive an additional decay rate, but it is not possible to apply the idea of \cite{GS87} to the VKG case, due to the massiveness of the Klein--Gordon field. It violates the strong Huygens' principle, and decay estimates based on the $ ( t - r )^{ - 1 } $ factor are irrelevant. To overcome this difficulty, Ha and Lee \cite{HL09} studied the VKG system by incorporating an additional damping term, which enabled the necessary decay estimates. In the present paper, we study the asymptotic behavior of small solutions to the undamped VKG system and resolve the aforementioned issues by employing the vector field method. We refer to \cite{Gl, GSc97, GS86, GS87, LS14, LS16, P05, P15, S04} for more details about the VM system and \cite{N24a, XZ18, XZ19} for recent results on the VKG system.

The vector field method, introduced by Klainerman in the 1980s, has been a powerful technique for studying wave equations and has recently been extended to the study of kinetic equations. In \cite{S16}, the vector field method was successfully adapted to the study of the Vlasov--Poisson system. The main idea was to lift the commutation vector fields to the tangent bundle in order to apply them in the phase space. The method was then extended to the Vlasov--Nordstr{\"o}m system \cite{FJS17, FJS17a}, to the VM system \cite{B20, B21, B22}, to the Vlasov--Yukawa system \cite{D22}, to the Einstein--Vlasov system \cite{BFJST21, FJS21}, and now will be applied to the VKG system in this paper. The spatial dimensions will be restricted to $ n \geq 4 $. In three dimensions we will need to consider the method of modified vector fields as in \cite{B20, FJS17a, S16} (see also \cite{N24a} for a different approach), and this will be our future project.

\subsection{Main result} \label{sec main}
In this paper we study the VKG system in dimensions $ n \geq 4 $. Let us write the VKG system as
\begin{align}
T_\phi f & = 0 , \label{V} \\
( \Box - 1 ) \phi & = \int_\bbrn f \, d v . \label{KG}
\end{align}
We are interested in small solutions to the VKG system \eqref{V}--\eqref{KG} with initial data given at the unit hyperboloid $ H_1 $. In the following theorem, and throughout the paper, the terminologies {\it regular} functions, {\it regular} solutions or {\it regular} distributions will be understood in the same way as in \cite{FJS17}. Details about the hyperboloidal foliation and the energies $ E_N ( \phi ) $, $ \hat{ E }_N ( f ) $ and $ \hat{ E }_{ N , 1 } ( f ) $ in the following theorem will be provided in Section \ref{sec prelim}.

\begin{theorem} \label{thm}
Let $ n \geq 4 $ and $ N \geq 5 n + 2 $. Let $ ( \phi , f ) $ be a regular solution of the VKG system \eqref{V}--\eqref{KG} with initial data given at the unit hyperboloid $ H_1 $. Then, there exists $ \varepsilon > 0 $ such that if initial data satisfy
\begin{align*}
E_N ( \phi ) ( 1 ) + \hat{ E }_{ N, 1 } ( f ) ( 1 ) + \hat{ E }_{ N + n } ( f ) ( 1 ) < \varepsilon ,
\end{align*}
then the corresponding solution satisfies the following estimates:
\begin{itemize}
\item[1.] The Klein--Gordon field satisfies for all $ \tau \geq 1 $,
\begin{align}
E_N ( \phi ) ( \tau ) \lesssim \varepsilon . \label{thm1}
\end{align}
\item[2.] The distribution function satisfies for all $ \tau \geq 1 $,
\begin{align}
\hat{ E }_{ N , 1 } ( f ) ( \tau ) \lesssim \varepsilon \tau^{ \delta } , \label{thm2}
\end{align}
where $ \delta = 0 $ for $ n \geq 5 $, and $ \delta = \varepsilon^{ \frac14 } $ for $ n = 4 $.
\item[3.] For any multi-index $ A $ satisfying $ N - n + 1 \leq | A | \leq N $, we have for all $ \tau \geq 1 $,
\begin{align}
\int_{ H_\tau } \frac{ t }{ \tau } \left( \int_{ \bbr^n } | \hat{ Z }_A f | \, d v \right)^2 d \mu_{ H_\tau } \lesssim \varepsilon^2 \tau^{ 2 \delta - n } , \label{thm3}
\end{align}
where $ \delta = 0 $ for $ n \geq 5 $, and $ \delta = \varepsilon^{ \frac14 } $ for $ n = 4 $.
\end{itemize}
Moreover, for each $ ( t , x ) $ in the future of the unit hyperboloid, we have the following estimates:
\begin{itemize}
\item[4.] For any multi-index $ A $ satisfying $ | A | \leq N - n $, we have
\begin{align}
& \int_{ \bbr^n } | \hat{ Z }_A f | \, d v \lesssim \frac{ \varepsilon }{ t^{ n - \frac{ 1 + \delta }{ 2 } } ( t - | x | )^{ \frac{ 1 - \delta }{ 2 } } } , \label{thm41} \\
& \int_{ \bbr^n } \frac{ 1 }{ v^0 } | \hat{ Z }_A f | \, d v \lesssim \frac{ \varepsilon }{ t^{ n - \delta } } , \label{thm42}
\end{align}
where $ \delta = 0 $ for $ n \geq 5 $, and $ \delta = \varepsilon^{ \frac14 } $ for $ n = 4 $.
\item[5.] For any multi-index $ A $ satisfying $ 0 \leq | A | \leq \lfloor N / 2 \rfloor - n $, we have
\begin{align}
\int_{ \bbr^n } | \hat{ Z }_A f | \, d v \lesssim \frac{ \varepsilon }{ t^{ n - \delta } } , \label{thm5}
\end{align}
where $ \delta = 0 $ for $ n \geq 5 $, and $ \delta = \varepsilon^{ \frac14 } $ for $ n = 4 $.
\item[6.] For any multi-index $ A $ satisfying $ 0 \leq | A | \leq N - \lfloor n / 2 \rfloor - 1 $, we have
\begin{align}
| Z_A \phi | & \lesssim \frac{ \varepsilon^{ \frac12 } }{ t^{ \frac{ n }{ 2 } } } , \label{thm61} \\
| \partial Z_A \phi | & \lesssim \frac{ \varepsilon^{ \frac12 } }{ t^{ \frac{ n - 1 }{ 2 } } ( t - | x | )^{ \frac12 } } , \label{thm62}
\end{align}
where $ \delta = 0 $ for $ n \geq 5 $, and $ \delta = \varepsilon^{ \frac14 } $ for $ n = 4 $.
\end{itemize}
\end{theorem}

\begin{remark}
Initial data are small with respect to the energies $ E_N ( \phi ) $, $ \hat{ E }_{ N , 1 } ( f ) $ and $ \hat{ E }_{ N + n } ( f ) $, but the corresponding solution is small with respect only to the energies $ E_N ( \phi ) $ and $ \hat{ E }_{ N , 1 } ( f ) $. The smallness of the initial value of the higher order energy $ \hat{ E }_{ N + n } ( f ) $ will only be used to control the homogeneous part of the commuted Vlasov equation. Then, this will be used again to obtain the estimate \eqref{thm3}. See Lemma \ref{lem homo} for more details.
\end{remark}

\begin{remark}
If we assume compact support on initial data, then we obtain a lower bound of the quantity $ t - | x | $, from which an additional decay in time is obtained for small solutions. This was a key idea in \cite{GS87} for proving the global existence of small solutions to the VM system (see also Lemmas 6.3.1 and 6.3.2 of \cite{Gl}). In this paper we do not assume compact support, but instead obtain an additional decay by using the hyperboloidal foliation and the vector field method.
\end{remark}

\subsection{Structure of the paper}
In Section \ref{sec prelim}, we introduce preliminary materials. Hyperboloidal foliation of the Minkowski spacetime and the Killing vector fields with their complete lifts are introduced in Sections \ref{sec hf} and \ref{sec Kf}, respectively. Energy-momentum tensors and energy densities are introduced in Section \ref{sec em}, and they are used in Section \ref{sec e} to define the energies for the distribution function and the Klein--Gordon field. In Section \ref{sec be}, we present basic estimates for the energy densities and obtain the Klainerman--Sobolev inequalities for the distribution function and the Klein--Gordon field. In Section \ref{sec commutation}, we derive commuted equations for the Vlasov and the Klein--Gordon equations. In Section \ref{sec proof}, we prove the main theorem. The proof is an application of the bootstrap argument. In Section \ref{sec bootsassumption}, we make bootstrap assumptions on the Klein--Gordon field and the distribution function. The assumptions are improved in Sections \ref{sec improvephi} and \ref{sec improvef} by assuming further that the distribution function satisfies a certain $ L^2 $-estimate. The $ L^2 $-estimate is verified in Section \ref{sec L2}, and this will complete the proof of Theorem \ref{thm}.

\section{Preliminaries} \label{sec prelim}

\subsection{Hyperboloidal foliation} \label{sec hf}
In this paper we will consider the hyperboloidal foliation of the $ ( n + 1 ) $-dimensional Minkowski spacetime. Let $ ( t , x ) \in \bbr^{ n + 1 } $ be the Cartesian coordinate system in the Minkowski spacetime. Then we define for each $ \tau \geq 1 $ the hyperboloid $ H_\tau $ by
\begin{align*}
H_\tau = \left\{ ( t , x ) \in \bbr^{ n + 1 } : t \geq | x | , \, \tau = \sqrt{ t^2 - | x |^2 } \right\} ,
\end{align*}
and obtain the hyperboloidal foliation of the Minkowski spacetime:
\begin{align*}
\bigcup_{ \tau \geq 1 } H_\tau = \left\{ ( t , x ) \in \bbr^{ n + 1 } : t \geq \sqrt{ 1 + | x |^2 } \right\} .
\end{align*}
With this foliation we introduce the pseudo-Cartesian coordinate system $ ( \tau , y ) \in \bigcup_{ \tau \geq 1 } H_\tau $ defined by
\begin{align}
\tau = \sqrt{ t^2 - | x |^2 } , \qquad y = x . \label{tau}
\end{align}
For each $ \tau \geq 1 $ we have the induced volume form on $ H_\tau $ defined by
\begin{align}
d \mu_{ H_\tau } = \frac{ \tau }{ t } r^{ n - 1 } d r \, d \sigma_{ \bbs^{ n - 1 } } , \label{dmuH_tau}
\end{align}
where $ ( r , \sigma ) \in [ 0 , \infty ) \times \bbs^{ n - 1 } $ is the usual spherical coordinate system in $ \bbr^n $ with the standard volume form $ d \sigma_{ \bbs^{ n - 1 } } $ of $ \bbs^{ n - 1 } $. To each $ H_\tau $ we have the future unit normal $ \nu_\tau $ given by
\begin{align}
\nu_\tau = \frac{ 1 }{ \tau }( t \partial_t + r \partial_r ) . \label{nutau}
\end{align}
We note that
\begin{align}
\partial_{ y^i } = \partial_{ x^i } + \frac{ x^i }{ t } \partial_t . \label{partialy0}
\end{align}
This will be used to prove the Klainermann--Sobolev inequalities in Propositions \ref{prop KSf} and \ref{prop KSphi}. For more details about the hyperboloidal foliation we refer to \cite{FJS17, LM14}.

In this paper we study the VKG system in the Minkowski spacetime with the hyperboloidal foliation. The distribution function and the Klein--Gordon field will be understood as functions defined in
\begin{align*}
\bigcup_{ 1 \leq \tau < T } H_\tau \times \bbr^n_v , \qquad \bigcup_{ 1 \leq \tau < T } H_\tau ,
\end{align*}
for some $ T \in ( 1 , \infty ] $, respectively, and their initial data will be given on $ H_1 \times \bbr^n_v $ and $ H_1 $, respectively. 

\subsection{Killing fields} \label{sec Kf}
In order to apply the vector field method to the VKG system we consider the Killing fields in the Minkowski spacetime and their commutations with the differential operators for the Vlasov and the Klein--Gordon equations. For the Vlasov equation we define the transport operators by
\begin{align*}
& T = v^0 \partial_t + v \cdot \nabla_x , \\
& T_\phi = T - v^0 \nabla_x \phi \cdot \nabla_v ,
\end{align*}
which we may call the free and the perturbed transport operators, respectively. Then the Vlasov equation can be written as
\begin{align*}
T_\phi f = 0 .
\end{align*}
The Klein--Gordon operator will be denoted by $ \Box - 1 $, so that the Klein--Gordon equation can be written as
\begin{align*}
( \Box - 1 ) \phi = \int_{ \bbr^n } f \, d v .
\end{align*}
Now, we introduce the following vector fields:
\begin{align*}
\bbp = \{ \partial_t , \, \partial_{ x^i } , \, \Omega_{ 0 i } , \, \Omega_{ i j } : i , j = 1 , \dots , n \} ,
\end{align*}
where $ \Omega_{ 0 i } $ and $ \Omega_{ i j } $ are defined by
\begin{align*}
\begin{aligned}
& \Omega_{ 0 i } = t \partial_{ x^i } + x^i \partial_t , & & 1 \leq i \leq n , \\
& \Omega_{ i j } = x^i \partial_{ x^j } - x^j \partial_{ x^i } , & & 1 \leq i < j \leq n ,
\end{aligned}
\end{align*}
which are the Killing fields in the $ ( n + 1 ) $-dimensional Minkowski spacetime. Then \eqref{partialy0} can be written as
\begin{align}
\partial_{ y^i } = \frac{ 1 }{ t } \Omega_{ 0 i } . \label{partialy}
\end{align}
We will denote an arbitrary member of $ \bbp $ by
\begin{align*}
Z_a \in \bbp ,
\end{align*}
for $ a = 1 , \dots , ( n + 1 ) ( n + 2 ) / 2 $. Any compositions of the members of $ \bbp $ will be written by using the multi-index notation. Let $ q \in \bbn $ and $ A = ( a_1 , \dots , a_q ) $ be a $ q $-tuple of integers such that
\begin{align*}
a_i \in \left\{ 1 , \dots , \frac{ ( n + 1 ) ( n + 2 ) }{ 2 } \right\} ,
\end{align*}
for $ i = 1 , \dots , q $. Then we define
\begin{align}
Z_A : = Z_{ a_1 } \cdots Z_{ a_q } , \label{Z_A}
\end{align}
which is a differential operator of order $ q $, and write $| A | = q $, which we call the length of $ A $. Finally, we denote by $ \bbp^{ | A | } $ the set of all the possible differential operators in the form of \eqref{Z_A}. Any $ q $-tuple of integers will be referred to as a multi-index of order $ q $, when it is used in the above manner.

We also need to consider the complete lifts of $ \bbp $:
\begin{align*}
\hat{ \bbp } = \left\{ \hat{ \partial }_t , \, \hat{ \partial }_{ x^i } , \, \hat{ \Omega }_{ i j } , \, \hat{ \Omega }_{ 0 i } : i , j = 1 , \dots , n \right\} ,
\end{align*}
defined by
\begin{align*}
& \hat{ \partial }_t = \partial_t , \\
& \hat{ \partial }_{ x^i } = \partial_{ x^i } , \\
& \hat{ \Omega }_{ 0 i } = \Omega_{ 0 i } + v^0 \partial_{ v^i } , \\
& \hat{ \Omega }_{ i j } = \Omega_{ i j } + v^i \partial_{ v^j } - v^j \partial_{ v^i } .
\end{align*}
We will denote an arbitrary member of $ \hat{ \bbp } $ by
\begin{align*}
\hat{ Z }_a \in \hat{ \bbp } ,
\end{align*}
for $ a = 1 , \dots , ( n + 1 ) ( n + 2 ) / 2 $, and will use the multi-index notation for higher order derivatives. Let $ A = ( a_1 , \dots , a_q ) $ be a multi-index. Then we define
\begin{align*}
\hat{ Z }_A : = \hat{ Z }_{ a_1 } \cdots \hat{ Z }_{ a_q } ,
\end{align*}
and the set $ \hat{ \bbp }^{ | A | } $ is also defined in a similar way. For more details about the complete lifts to vector fields we refer to the Appendix of \cite{FJS17}.

\subsection{Energy-momentum tensor} \label{sec em}
We also need to introduce the energy-momentum tensors for the distribution function and the Klein--Gordon field. We define the energy-momentum tensor $ T^\rV $ for $ f $ as
\begin{align*}
& T^\rV_{ \mu \nu } [ f ] : = \int_{ \bbr^n } \frac{ v_\mu v_\nu }{ v^0 } f \, d v ,
\end{align*}
and define the energy density as
\begin{align*}
\hat{ e } ( f ) : = T^\rV [ f ] ( \partial_t , \nu_\tau ) = - \int_{ \bbr^n } \frac{ v_\mu x^\mu }{ \tau } f \, d v ,
\end{align*}
where $ \nu_\tau $ is the future unit normal \eqref{nutau} to $ H_\tau $. For the Klein--Gordon field we define the energy-momentum tensor $ T^\rK $ as
\begin{align*}
T^\rK [ \phi ] : = d \phi \otimes d \phi - \frac12 ( \eta ( \partial \phi , \partial \phi ) + \phi^2 ) \eta ,
\end{align*}
where $ \eta $ is the Minkowski metric, and $ \partial \phi $ denotes $ ( \partial_t \phi , \nabla_x \phi ) $. We define the energy density as
\begin{align*}
e  ( \phi ) : = T^\rK [ \phi ] ( \partial_t , \nu_\tau ) ,
\end{align*}
where $ \nu_\tau $ is the future unit normal \eqref{nutau} to $ H_\tau $.

\subsection{Energies} \label{sec e}
With the Killing fields in Section \ref{sec Kf} and the energy densities in Section \ref{sec em} we define the energy for the distribution function as
\begin{align}
\hat{ E }_N ( f ) & : = \sum_{ | A | \leq N } \int_{ H_\tau }  \hat{ e } ( | \hat{ Z }_A f | ) \, d \mu_{ H_\tau } , \label{fnorm} \\
\hat{ E }_{ N , 1 } ( f ) & : = \sum_{ | A | \leq \lfloor N / 2 \rfloor } \int_{ H_\tau } \hat{ e } ( v^0 | \hat{ Z }_A f | ) \, d \mu_{ H_\tau } + \sum_{ \lfloor N / 2 \rfloor + 1 \leq | A | \leq N } \int_{ H_\tau } \hat{ e } ( | \hat{ Z }_A f | ) \, d \mu_{ H_\tau } , \label{fnormw}
\end{align}
and the energy for the Klein--Gordon field as
\begin{align}
E_N ( \phi ) : = \sum_{ | A | \leq N } \int_{ H_\tau } e ( Z_A \phi ) \, d \mu_{ H_\tau } . \label{phinorm}
\end{align}
Here, $ \lfloor \cdot \rfloor $ denotes the floor function, so that $ \lfloor N / 2 \rfloor $ is the greatest integer satisfying $ \lfloor N / 2 \rfloor \leq N / 2 $. The summations are taken over all the possible multi-indices $ A $ satisfying $ | A | \leq N $ in \eqref{fnorm} and \eqref{phinorm}, and $ | A | \leq \lfloor N / 2 \rfloor $ and $ \lfloor N / 2 \rfloor + 1 \leq | A | \leq N $ in \eqref{fnormw}.


%
\section{Basic estimates} \label{sec be}
In this section we introduce the basic estimates which are necessary for the application of the vector field method to the VKG system. In Section \ref{sec be1} we consider the estimates of the energy densities. The estimates in Section \ref{sec be1} will be extended to the Klainerman--Sobolev inequalities in Section \ref{sec be2}.

\subsection{Estimates of the energy densities} \label{sec be1}

\subsubsection{Energy density for the distribution function} \label{sec basicf}
Let us first consider the energy density $ \hat{ e } $ for the distribution function. Lemma \ref{lem chiintf} below shows that $ \hat{ e } $ can be used to control the $ L^1_v $-norm of $ f $. In Lemma \ref{lem OmegaOmegahat} we show that the integration over $ \bbr^n_v $ commutes with the vector fields of $ \bbp $ or $ \hat{ \bbp } $. Lemma \ref{lem intchif} is an application of the divergence theorem to the transport equation with an inhomogeneous term.

\begin{lemma} \label{lem chiintf}
For $ t \geq r $ we have the following inequalities:
\begin{align*}
\hat{ e } ( | f | ) & \geq \frac{ t }{ 2 \tau } \int_\bbrn \frac{ 1 }{ v^0 } | f | \, d v , \\
\hat{ e } ( | f | ) & \geq \frac{ \tau }{ 2 ( t + r ) } \int_\bbrn v^0 | f | \, d v , \\
\hat{ e } ( | f | ) & \geq \int_\bbrn | f | \, d v .
\end{align*}
\end{lemma}
\begin{proof}
We refer to Lemma 2.11 and Remark 2.12 of \cite{FJS17} for the proof of this lemma.
\end{proof}

\begin{lemma} \label{lem OmegaOmegahat}
For $ k = 0 , 1 $, we have the following identities:
\begin{enumerate}
\item For $ \mu = 0 , \dots , n $,
\begin{align*}
\partial_{ x^\mu } \int_\bbrn \frac{ 1 }{ ( v^0 )^k } f \, d v = \int_\bbrn \frac{ 1 }{ ( v^0 )^k } \hat{ \partial }_{ x^\mu } f \, d v .
\end{align*}
\item For $ 1 \leq i < j \leq n $,
\begin{align*}
\Omega_{ i j } \int_\bbrn \frac{ 1 }{ ( v^0 )^k } f \, d v = \int_\bbrn \frac{ 1 }{ ( v^0 )^k } \hat{ \Omega }_{ i j } f \, d v .
\end{align*}
\item For $ i = 1 , \dots , n $,
\begin{align*}
\Omega_{ 0 i } \int_\bbrn \frac{ 1 }{ ( v^0 )^k } f \, d v = \int_\bbrn \frac{ 1 }{ ( v^0 )^k } \hat{ \Omega }_{ 0 i } f \, d v + ( 1 - k ) \int_\bbrn \frac{ v^i }{ ( v^0 )^{ k + 1 } } f \, d v .
\end{align*}
\item All the above identities hold with $ f $ replaced by $ | f | $.
\end{enumerate}
\end{lemma}
\begin{proof}
The first and the second identities are clear, and the third one is given by
\begin{align*}
\Omega_{ 0 i } \int \frac{ 1 }{ ( v^0 )^k } f \, d v & = \int \frac{ 1 }{ ( v^0 )^k } \hat{ \Omega }_{ 0 i } f \, d v - \int \frac{ 1 }{ ( v^0 )^{ k - 1 } } \partial_{ v^i } f \, d v \\
& = \int \frac{ 1 }{ ( v^0 )^k } \hat{ \Omega }_{ 0 i } f \, d v + ( 1 - k ) \int \frac{ v^i }{ ( v^0 )^{ k + 1 } } f \, d v .
\end{align*}
This completes the proof of the lemma.
\end{proof}

\begin{lemma} \label{lem intchif}
Let $ f $ and $ h $ be regular functions satisfying $ T_\phi f = h $ defined in $ \bigcup_{ 1 \leq \tau < T } H_\tau \times \bbr^n_v $. Then we have
\begin{align*}
& \int_{ H_\tau } \hat{ e } ( | f |) \, d \mu_{ H_\tau } \leq \int_{ H_1 } \hat{ e } ( | f | ) \, d \mu_{ H_1 } + \int_1^\tau \int_{ H_\lambda } \int_\bbrn \left( | h | + | \nabla_x \phi | | f | \right) d v \, d \mu_{ H_\lambda } d \lambda .
\end{align*}
\end{lemma}
\begin{proof}
Note that the divergence of $ T^\rV $ is given by
\begin{align*}
\partial^\mu T^\rV_{ \mu 0 } & = \int_\bbrn \frac{ v_\mu v_0 }{ v^0 } \partial^\mu f \, d v \\
& = - \int_\bbrn T f \, d v .
\end{align*}
Integrating this over $ \bigcup_{ 1 \leq \lambda \leq \tau } H_\lambda $ we obtain from the right hand side
\begin{align*}
- \int_1^\tau \int_{ H_\lambda } \int_\bbrn T f \, d v \, d \mu_{ H_\lambda } d \lambda & = - \int_1^\tau \int_{ H_\lambda } \int_\bbrn h + v^0 \nabla_x \phi \cdot \nabla_v f \, d v \, d \mu_{ H_\lambda } d \lambda \\
& = - \int_1^\tau \int_{ H_\lambda } \int_\bbrn h - \left( \frac{ v }{ v^0 } \cdot \nabla_x \phi \right) f \, dv \, d \mu_{H_\lambda} d \lambda ,
\end{align*}
and from the left hand side
\begin{align*}
\int_1^\tau \int_{ H_\lambda } \partial^\mu T^\rV_{ \mu 0 } \, d \mu_{H_\lambda} d \lambda & = - \int_{ H_\tau } T^\rV_{ \mu 0 } ( \nu_\tau )^\mu d \mu_{ H_\tau } + \int_{ H_1 } T^\rV_{ \mu 0 } ( \nu_1 )^\mu d \mu_{ H_1 } \\
& = - \int_{ H_\tau } \hat{ e } ( f ) \, d \mu_{ H_\tau } + \int_{ H_1 } \hat{ e } ( f ) \, d \mu_{ H_1 } ,
\end{align*}
where $ \nu_\tau $ and $ \nu_1 $ are the future unit normals to $ H_\tau $ and $ H_1 $, respectively. Now, by the same argument as in Lemma 4.3 of \cite{FJS17}, we obtain the desired inequality for $ | f | $.
\end{proof}

\subsubsection{Energy density for the Klein--Gordon field} \label{sec basicphi}
Applying the future unit normal $ \nu_\tau $ given by \eqref{nutau} to the definition of the energy density $ e $ for the Klein--Gordon field we obtain the following estimate:
\begin{align*}
e ( \phi ) & = \frac{ t }{ 2 \tau } \left( (\partial_t \phi )^2 + | \nabla_x \phi |^2 + \phi^2 \right) + \frac{ r }{ \tau } ( \partial_t \phi ) ( \partial_r \phi ) \\
& = \frac{ t }{ 2 \tau } \phi^2 + \frac{ r }{ 2 \tau } ( \partial_t \phi + \partial_r \phi )^2 + \frac{ t - r }{ 2 \tau } ( \partial_t \phi )^2 + \frac{ t }{ 2 \tau } | \nabla_x \phi |^2 - \frac{ r }{ 2 \tau } ( \partial_r \phi )^2 \\
& \geq \frac{ t }{ 2 \tau } \phi^2 + \frac{ t - r }{ 2 \tau } ( \partial_t \phi )^2 + \frac{ t - r }{ 2 \tau } | \nabla_x \phi |^2 .
\end{align*}
This implies that
\begin{align}
e ( \phi ) \geq \frac{ t }{ 2 \tau } \phi^2 + \frac{ \tau }{ 2 ( t + r ) } | \partial \phi |^2 , \label{chiphiphi}
\end{align}
by the definition \eqref{tau} of $ \tau $. This inequality shows that the energy \eqref{phinorm} is basically an $ L^2 $-norm, while the energies \eqref{fnorm} and \eqref{fnormw} are $ L^1 $-norms (see Lemma \ref{lem chiintf}).

The following lemma is an application of the divergence theorem to the Klein--Gordon equation with an inhomogeneous term.

\begin{lemma} \label{lem chiphiest}
Let $ \phi $ and $ h $ be regular functions satisfying $ ( \Box -1 ) \phi = h $ defined in $ \bigcup_{ 1 \leq \tau < T } H_\tau $. Then we have
\begin{align*}
\int_{ H_\tau } e ( \phi ) \, d \mu_{ H_\tau } = \int_{ H_1 } e ( \phi ) \, d \mu_{ H_1 } - \int_1^\tau \int_{ H_\lambda } h \partial_t \phi \, d \mu_{ H_\lambda } d \lambda .
\end{align*}
\end{lemma}
\begin{proof}
Divergence of the energy-momentum tensor for the Klein--Gordon field is given by
\begin{align*}
\partial^\mu T^\rK_{ \mu \nu } & = ( \partial^\mu \partial_\mu \phi ) ( \partial_\nu \phi ) + ( \partial_\mu \phi ) ( \partial^\mu \partial_\nu \phi ) - \frac12 \left( 2 ( \partial^\mu \partial^\kappa \phi ) ( \partial_\kappa \phi ) + 2 \phi ( \partial^\mu \phi ) \right) \eta_{ \mu \nu } \\
& = ( \phi + h ) \partial_\nu \phi - \phi ( \partial_\nu \phi ) \\
& = h ( \partial_\nu \phi ) .
\end{align*}
Integrating the zeroth component over $ \bigcup_{ 1 \leq \lambda \leq \tau } H_\lambda $ we obtain from the right hand side
\begin{align*}
\int_1^\tau \int_{ H_\lambda } h ( \partial_t \phi ) \, d \mu_{ H_\lambda } d \lambda ,
\end{align*}
and from the left hand side
\begin{align*}
\int_1^\tau \int_{ H_\lambda } \partial^\mu T^\rK_{ \mu 0 } \, d \mu_{ H_\lambda } d \lambda & = - \int_{ H_\tau } T^\rK_{ \mu 0 } ( \nu_\tau )^\mu d \mu_{ H_\tau } + \int_{ H_1 } T^\rK_{ \mu 0 } ( \nu_1 )^\mu d \mu_{ H_1 } \\
& = - \int_{ H_\tau } e ( \phi ) \, d \mu_{ H_\tau } + \int_{ H_1 } e ( \phi ) \, d \mu_{ H_1 } ,
\end{align*}
where $ \nu_\tau $ and $ \nu_1 $ are the future unit normals to $ H_\tau $ and $ H_1 $, respectively. This completes the proof.
\end{proof}

\subsection{Klainerman--Sobolev inequalities} \label{sec be2}
In order to apply the vector field method to the VKG system we need to derive the Klainerman--Sobolev inequalities for the distribution function and the Klein--Gordon field. We introduce the Klainerman--Sobolev inequalities for the distribution function in Proposition \ref{prop KSf} and the ones for the Klein--Gordon field in Proposition \ref{prop KSphi}.

\begin{proposition} \label{prop KSf}
Let $ f $ be a regular distribution function defined in $ \bigcup_{ 1 \leq \tau < T } H_\tau \times \bbrn $ for some $ T \in ( 1 , \infty ] $. Then we have
\begin{align}
\int_\bbrn \frac{ 1 }{ ( v^0 )^k } | f | ( t , x , v ) \, d v & \lesssim \frac{ 1 }{ t^{ n - 1 + k } \tau^{ 1 - k } } \hat{ E }_n ( f ) ( \tau ) , \qquad k = 0 , 1 , \label{KSf} \\
\int_\bbrn | f | ( t , x , v ) \, d v & \lesssim \frac{ 1 }{ t^n } \sum_{ | A | \leq n } \int_{ H_\tau } \hat{ e } ( v^0 | \hat{ Z }_A f | ) \, d \mu_{ H_\tau } , \label{KSfimprove}
\end{align}
for $ t $, $ x $ and $ \tau $ satisfying \eqref{tau}.
\end{proposition}
\begin{proof}
We basically follow the proof of Theorem 8 of \cite{FJS17}. Let $ ( t , x ) \in \bigcup_{ 1 \leq \tau < T } H_\tau $ be fixed with $ \tau = \sqrt{ t^2 - r^2 } $ and $ r = | x | $, and consider the function defined by
\begin{align*}
\psi ( y ) : = \int_{ \bbr^n } \frac{ 1 }{ ( v^0 )^k } | f | ( \tau , x + t y , v ) \, d v \qquad \mbox{for} \quad | y | \leq \frac18 ,
\end{align*}
where the integrand is written in the pseudo-Cartesian coordinate system. By the one dimensional Sobolev inequality we obtain
\begin{align*}
| \psi ( 0 ) | \lesssim \int_{ | y^1 | \leq \frac{ 1 }{ 8 \sqrt{ n } } } \left( | \psi | ( y^1 , 0 , \dots , 0 ) + | \partial_{ y^1 } \psi | ( y^1 , 0 , \dots , 0 ) \right) d y^1 .
\end{align*}
Note that the point $ ( \tau , x + t y ) $ in the pseudo-Cartesian coordinates can be written as $ ( s , x + t y ) $ in the usual Cartesian coordinate system, where
\begin{align*}
s = \sqrt{ \tau^2 + | x + t y |^2 } ,
\end{align*}
so that we obtain by \eqref{partialy}
\begin{align*}
\partial_{ y^i } \psi ( y ) & = \partial_{ y^i } \int_\bbrn \frac{ 1 }{ ( v^0 )^k } | f | ( \tau , x + t y , v ) \, d v \\
& = \frac{ t }{ s } \Omega_{ 0 i } \int_\bbrn \frac{ 1 }{ ( v^0 )^k } | f | ( s , x + t y , v ) \, d v .
\end{align*}
Since we have for $ r \leq t / 2 $,
\begin{align}
s = \sqrt{ \tau^2 + | x + t y |^2 } \geq \tau = \sqrt{ t^2 - r^2 } \geq \frac{ \sqrt{ 3 } }{ 2 } t , \label{st1}
\end{align}
and for $ r \geq t / 2 $,
\begin{align}
s = \sqrt{ \tau^2 + | x + t y |^2 } \geq | x + t y | \geq r - t | y | \geq \frac38 t , \label{st2}
\end{align}
we obtain
\begin{align*}
| \partial_{ y^1 } \psi ( y ) | \lesssim \left| \Omega_{ 0 1 } \int_\bbrn \frac{ 1 }{ ( v^0 )^k } | f | ( s , x + t y , v ) \, d v \right| .
\end{align*}
Consequently we obtain by Lemma \ref{lem OmegaOmegahat}
\begin{align*}
| \psi ( 0 ) | \lesssim \sum_{ | A | \leq 1 } \int_{ | y^1 | \leq \frac{ 1 }{ 8 \sqrt{ n } } } \int_\bbrn \frac{ 1 }{ ( v^0 )^k } | \hat{ Z }_A f | \left( s , x + t ( y^1 , 0 , \dots , 0 ) , v \right) d v \, d y^1 .
\end{align*}
Inductively we have
\begin{align*}
| \psi ( 0 ) | \lesssim \sum_{ | A | \leq n } \int_{ | y | \leq \frac18 } \int_\bbrn \frac{ 1 }{ ( v^0 )^k } | \hat{ Z }_A f | ( s , x + t y , v ) \, d v \, d y .
\end{align*}
By the change of variables we obtain
\begin{align*}
| \psi ( 0 ) | & \lesssim \frac{ 1 }{ t^n } \sum_{ | A | \leq n } \int_\bbrn \int_\bbrn \frac{ 1 }{ ( v^0 )^k } | \hat{ Z }_A f | \left( \sqrt{ \tau^2 + | z |^2 } , z , v \right) d v \, d z \\
& \lesssim \frac{ 1 }{ t^{ n - 1 } \tau } \sum_{ | A | \leq n } \int_{ H_\tau } \int_\bbrn \frac{ 1 }{ ( v^0 )^k } | \hat{ Z }_A f | ( \tau , z , v ) \, d v \, d \mu_{ H_\tau } \\
& \lesssim \frac{ 1 }{ t^{ n - 1 + k } \tau^{ 1 - k } } \sum_{ | A | \leq n } \int_{ H_\tau } \hat{ e } ( | \hat{ Z }_A f | ) \, d \mu_{ H_\tau } ,
\end{align*}
where we used \eqref{dmuH_tau} and Lemma \ref{lem chiintf}. Note that the point $ ( \tau , x ) $ in the pseudo-Cartesian coordinates is written as $ ( t , x ) $ in the usual Cartesian coordinates. Hence we have
\begin{align*}
\psi ( 0 ) = \int_\bbrn \frac{ 1 }{ ( v^0 )^k } | f | ( t , x , v ) \, d v ,
\end{align*}
and this completes the proof of the first inequality. The second inequality is derived from the first inequality with $ f $ replaced by $ v^0 f $ with $ k = 1 $.
\end{proof}


%
\begin{proposition} \label{prop KSphi}
Let $ \phi $ be a regular function defined in $ \bigcup_{ 1 \leq \tau < T } H_\tau $ for some $ T \in ( 1 , \infty ] $. Then we have
\begin{align}
| \phi | ( t , x ) & \lesssim \frac{ 1 }{ t^\frac{ n }{ 2 } } E_{ \lfloor n / 2 \rfloor +1 }^{ \frac12 } ( \phi ) ( \tau ) , \label{KSphi} \\
| \partial \phi | ( t , x ) & \lesssim \frac{ 1 }{ t^{ \frac{ n }{ 2 } - 1 } \tau } E_{ \lfloor n / 2 \rfloor +1 }^{ \frac12 } ( \phi ) ( \tau ) , \label{KSphiimprove}
\end{align}
for $ t $, $ x $ and $ \tau $ satisfying \eqref{tau}.
\end{proposition}
\begin{proof}
We first note that by \eqref{chiphiphi}
\begin{align}
E_m ( \phi ) ( \tau ) \geq \frac{ t }{ 2 \tau } \sum_{ | A | \leq m } \int_{ H_\tau } | Z_A \phi |^2 d \mu_{ H_\tau } . \label{EphiZphi}
\end{align}
On the other hand, let $ ( t , x ) $ be fixed, and consider
\begin{align*}
\psi ( y ) : = \phi ( \tau , x + t y ) ,
\end{align*}
where we used the pseudo-Cartesian coordinate system for $ \phi $. For simplicity, we write 
\begin{align*}
s = \sqrt{ \tau^2 + | x + t y |^2 } ,
\end{align*}
as in Proposition \ref{prop KSf}, and obtain by \eqref{partialy} and the Sobolev inequality
\begin{align*}
| \psi ( 0 ) |^2 & \lesssim \sum_{ k \leq \lfloor n / 2 \rfloor +1 } \int_{ | y | \leq \delta } | \partial_y^k \psi ( y ) |^2 d y \\
& \lesssim \sum_{ | A | \leq \lfloor n / 2 \rfloor +1 } \int_{ | y | \leq \delta } \frac{ t^{ 2 k } }{ s^{ 2 k } } \left| Z_A \phi ( s , x + t y ) \right|^2 d y .
\end{align*}
In a similar way to \eqref{st1}--\eqref{st2}, we have $ t \lesssim s $, and since $ \psi ( 0 ) = \phi ( t , x ) $, we obtain
\begin{align*}
| \phi ( t , x ) |^2 & \lesssim \sum_{ | A | \leq \lfloor n / 2 \rfloor + 1 } \int_{ | y | \leq \delta } \left| Z_A \phi ( s , x + t y ) \right|^2 d y \\
& \lesssim \frac{ 1 }{ t^n } \sum_{ | A | \leq \lfloor n / 2 \rfloor + 1 } \int_{ \bbr^n } \left| Z_A \phi \left( \sqrt{ \tau^2 + | z |^2 } , z \right) \right|^2 d z \\
& \lesssim \frac{ 1 }{ t^{ n - 1 } \tau } \sum_{ | A | \leq \lfloor n / 2 \rfloor +1 } \int_{ H_\tau } | Z_A \phi |^2 d \mu_{ H_\tau } \\
& \lesssim \frac{ 1 }{ t^n } E_{ \lfloor n / 2 \rfloor + 1 } ( \phi ) ,
\end{align*}
where we used \eqref{dmuH_tau} and \eqref{EphiZphi}. If we consider $ \psi = \partial \phi $ and use the inequality for $ \partial \phi $ in \eqref{chiphiphi}, then we obtain the second result. 
\end{proof}

\section{Commuted equations} \label{sec commutation}
We need to estimate higher order derivatives of the distribution function and the Klein--Gordon field in order to control the energies \eqref{fnorm}, \eqref{fnormw} and \eqref{phinorm}. Recall that the Vlasov equation can be written as
\begin{align*}
T_\phi f = 0 .
\end{align*}
Then, the equation for $ \hat{ Z }_A f $ with $ | A | \geq 1 $ will be given in the form of
\begin{align*}
T_\phi \hat{ Z }_A f = g ,
\end{align*}
where we need commutation formulas for $ T_\phi $ with $ \hat{ Z }_A $ to find the right hand side. The Klein--Gordon equation can be written as
\begin{align*}
( \Box - 1 ) \phi & = \int_\bbrn f \, d v .
\end{align*}
It is well-known that the operator $ \Box - 1 $ commutes with $ Z_A $ (for instance see \cite{K85, K93}), but we already have an inhomogeneous term on the right hand side, so that we will obtain
\begin{align*}
( \Box - 1 ) Z_A \phi = h ,
\end{align*}
where we need to use Lemma \ref{lem OmegaOmegahat} to find the right hand side.

In this section we will derive the equations for $ \hat{ Z }_A f $ and $ Z_A \phi $. In Section \ref{sec commutationf} we will consider the commutation formulas for the transport operator $ T_\phi $ with the vector fields $ \hat{ Z }_A $ for $ | A | \geq 1 $ to derive the equation for $ \hat{ Z }_A f $. In Section \ref{sec commutationphi} we apply Lemma \ref{lem OmegaOmegahat} to obtain the equation for $ Z_A \phi $.

\subsection{Vlasov equation} \label{sec commutationf}
We need to consider the commutation formulas for the transport operator $ T_\phi $ with $ \hat{ Z }_A $ in order to derive the equation for $ \hat{ Z }_A f $. Let us first consider the commutation formulas for the free transport operator $ T $ with single vector fields $ \hat{ Z }_a \in \hat{ \bbp } $. One can easily obtain the following lemma.

\begin{lemma} \label{lem basic}
The following commutation rules hold:
\begin{enumerate}
\item For any $ \hat{ Z }_a \in \hat{ \bbp } $, the free transport operator $ T $ satisfies
\begin{align*}
\left[ T , \hat{ Z }_a \right] = 0 .
\end{align*}
\item For any $ \hat{ Z }_a , \hat{ Z }_b \in \hat{ \bbp } $, there exist constant coefficients $ C_{ a b }^c $ such that
\begin{align*}
\left[ \hat{ Z }_a , \hat{ Z }_b \right] = \sum_{ \hat{ Z }_c \in \hat{ \bbp } } C_{ a b }^c \hat{ Z }_c .
\end{align*}
\item Concerning the derivatives $ \partial_{ v^i } $ we have
\begin{align*}
& \left[ \hat{ \partial }_t , \partial_{ v^i } \right] = 0 , \\
& \left[ \hat{ \partial }_{ x^i } , \partial_{ v^j } \right] = 0 ,\\
& \left[ \hat{ \Omega }_{ 0 i } , \partial_{ v^j } \right] = - \frac{ v^j }{ v^0 } \partial_{ v^i } , \\
& \left[ \hat{ \Omega }_{ i j } , \partial_{ v^k } \right] = - \delta^i_k \partial_{ v^j } + \delta^j_k \partial_{ v^i } , \\
& \left[ \hat{ \Omega }_{ 0 i } , v^j \partial_{ v^j } \right] = \frac{ 1 }{ v^0 } \partial_{ v^i } , \\
& \left[ \hat{ \Omega }_{ i j } , v^k \partial_{ v^k } \right] = 0 ,
\end{align*}
for any $ i , j , k = 1 , \dots , n $.
\end{enumerate}
\end{lemma}
\begin{proof}
The lemma is proved by direct computations. We refer to \cite{FJS17, S16} for more details.
\end{proof}

It is more involved to derive commutation formulas for the perturbed transport operator $ T_\phi $ with $ \hat{ Z }_A $ of order $ | A | \geq 1 $. For the simplicity of notations let us denote by $ \cP $ and $ \cQ $ the sets of polynomials defined by
\begin{align*}
\cP & = \mbox{the set of polynomials of $ v^1 / v^0 $, $ v^2 / v^0 $ and $ v^3 / v^0 $,} \\
\cQ & = \mbox{the set of polynomials in the form of $ P_0 + t P_1 + \sum_{ k = 1 }^n x^k P_{ k + 1 } $ with $ P_0 , \dots , P_{ n + 1 } \in \cP $.}
\end{align*}
In other words, $ \cP $ is the set of polynomials of $ v^j / v^0 $ with constant coefficients, while $ \cQ $ is the set of first order polynomials of $ t $ and $ x^j $ with coefficients in $ \cP $.

\begin{lemma} \label{lem TphiZhat1}
For any $ \hat{ Z }_a \in \hat{ \bbp } $, we have
\begin{align}
\left[ T_\phi , \hat{ Z }_a \right] = v^0 \sum_{ | B | \leq 1 } \sum_{ \substack{ 0 \leq \mu \leq n \\ 1 \leq i \leq n } } P^{ \mu B i }_{ a } \partial_{ x^\mu } ( Z_B \phi ) \partial_{ v^i } , \label{TphiZhat1}
\end{align}
for some $ P^{ \mu B i }_{ a } \in \cP $. 
\end{lemma}
\begin{proof}
Recall that $ T_\phi = T - v^0 \nabla_x \phi \cdot \nabla_v $. By Lemma \ref{lem basic} we only need to consider
\begin{align*}
\left[ T_\phi , \hat{ Z }_a \right] = \left[ - v^0 \nabla_x \phi \cdot \nabla_v , \hat{ Z }_a \right] .
\end{align*}
Since $ \hat{ \partial }_t = \partial_t $ and $ \hat{ \partial }_{ x^i } = \partial_{ x^i } $, we easily obtain
\begin{align}
& \left[ - v^0 \nabla_x \phi \cdot \nabla_v , \hat{ \partial }_t \right] = v^0 \nabla_x ( \partial_t \phi ) \cdot \nabla_v , \label{TphiZhat11} \\
& \left[ - v^0 \nabla_x \phi \cdot \nabla_v , \hat{ \partial }_{ x^i } \right] = v^0 \nabla_x ( \partial_{ x^i } \phi ) \cdot \nabla_v . \label{TphiZhat12}
\end{align}
For $ \hat{ Z }_a = \hat{ \Omega }_{ i j } $, we compute as follows:
\begin{align}
& \left[ - v^0 \nabla_x \phi \cdot \nabla_v , \hat{ \Omega }_{ i j } \right] \nonumber \\
& = \left[ - v^0 \nabla_x \phi \cdot \nabla_v , \Omega_{ i j } \right] + \left[ - v^0 \nabla_x \phi \cdot \nabla_v , v^i \partial_{ v^j } - v^j \partial_{ v^i } \right] \nonumber \\
& = v^0 \Omega_{ i j } ( \nabla_x \phi ) \cdot \nabla_v - v^0 \nabla_x \phi \cdot \left( ( \nabla_v v^i ) \partial_{ v^j } - ( \nabla_v v^j ) \partial_{ v^i } \right) + \left( v^i \frac{ v^j }{ v^0 } - v^j \frac{ v^i }{ v^0 } \right) \nabla_x \phi \cdot \nabla_v \nonumber \\
& = v^0 \nabla_x ( \Omega_{ i j } \phi ) \cdot \nabla_v . \label{TphiZhat13}
\end{align}
Similarly, for $ \hat{ Z }_a = \hat{ \Omega }_{ 0 i } $, we have
\begin{align}
& \left[ - v^0 \nabla_x \phi \cdot \nabla_v , \hat{ \Omega }_{ 0 i } \right] \nonumber \\
& = v^0 \nabla_x ( \Omega_{ 0 i } \phi ) \cdot \nabla_v - v^0 ( \partial_t \phi ) \partial_{ v^i } - ( v \cdot \nabla_x \phi ) \partial_{ v^i } + v^i \nabla_x \phi \cdot \nabla_v . \label{TphiZhat14}
\end{align}
All the quantities \eqref{TphiZhat11}--\eqref{TphiZhat14} can be written as in \eqref{TphiZhat1} for some polynomials $ P^{ \mu B i }_{ a } \in \cP $. This completes the proof of the lemma.
\end{proof}

\begin{lemma} \label{lem TphiZhat2}
For any $ \hat{ Z }_a \in \hat{ \bbp } $, we have
\begin{align*}
\left[ T_\phi , \hat{ Z }_a \right] = \sum_{ \substack{ | B | \leq 1 \\ | C | = 1 } } \sum_{ 0 \leq \mu \leq n } Q^{ \mu B C }_{ a } \partial_{ x^\mu } ( Z_B \phi ) \hat{ Z }_C ,
\end{align*}
for some $ Q^{ \mu B C }_a \in \cQ $.
\end{lemma}
\begin{proof}
Concerning the derivatives $ \partial_{ v^i } $ we can use the following decomposition:
\begin{align}
\partial_{ v^i } = \frac{ 1 }{ v^0 } \hat{ \Omega }_{ 0 i } - \frac{ 1 }{ v^0 } ( t \partial_{ x^i } + x^i \partial_t ) . \label{partialvOmegahat}
\end{align}
Note that $ \hat{ \Omega }_{ 0 i } $, $ \partial_t $ and $ \partial_{ x^i } $ belong to $ \hat{ \bbp } $. Now, we apply \eqref{partialvOmegahat} to Lemma \ref{lem TphiZhat1} to obtain the desired result.
\end{proof}

\begin{lemma} \label{lem TphiZhat3}
For any $ \hat{ Z }_A \in \hat{ \bbp }^{ | A | } $ with $ | A | \geq 1 $, we have
\begin{align*}
\left[ T_\phi , \hat{ Z }_A \right] = \sum_{ \substack{ | B | + | C | \leq | A | + 1 \\ 1 \leq | C | \leq | A | } } \sum_{ 0 \leq \mu \leq n } Q^{ \mu B C }_{ A } \partial_{ x^\mu } ( Z_B \phi ) \hat{ Z }_C ,
\end{align*}
for some $ Q^{ \mu B C }_{ A } \in \cQ $.
\end{lemma}
\begin{proof}
The lemma holds in the case $ | A | = 1 $ by Lemma \ref{lem TphiZhat2}. Suppose that the lemma holds for some $ | A | = k $, and compute
\begin{align*}
& \left[ T_\phi , \hat{ Z }_a \hat{ Z }_A \right] \\
& = \left[ T_\phi , \hat{ Z }_a \right] \hat{ Z }_A + \hat{ Z }_a \left[ T_\phi , \hat{ Z }_A \right] \\
& = \left( \sum_{ \substack{ | B | \leq 1 \\ | C | = 1 } } \sum_{ 0 \leq \mu \leq n } Q^{ \mu B C }_{ a } \partial_{ x^\mu } ( Z_B \phi ) \hat{ Z }_C \right) \hat{ Z }_A + \hat{ Z }_a \left( \sum_{ \substack{ | B | + | C | \leq | A | + 1 \\ 1 \leq | C | \leq | A | } } \sum_{ 0 \leq \mu \leq n } Q^{ \mu B C }_{ A } \partial_{ x^\mu } ( Z_B \phi ) \hat{ Z }_C \right) \\
& = \sum_{ \substack{ | B | \leq 1 \\ | C | = | A | + 1 } } \sum_{ 0 \leq \mu \leq n } Q^{ \mu B C }_{ a } \partial_{ x^\mu } ( Z_B \phi ) \hat{ Z }_C + \sum_{ \substack{ | B | + | C | \leq | A | + 1 \\ 1 \leq | C | \leq | A | } } \sum_{ 0 \leq \mu \leq n } \hat{ Z }_a \left( Q^{ \mu B C }_{ A } \partial_{ x^\mu } ( Z_B \phi ) \hat{ Z }_C \right) ,
\end{align*}
where $ Q^{ \mu B C }_{ a } $ and $ Q^{ \mu B C }_{ A } $ are polynomials in $ \cQ $, which may change from line to line. Let us compute the $ \hat{ Z }_a $ derivative in the second term. Note that the polynomial $ Q^{ \mu B C }_{ A } $ can be written as
\begin{align*}
Q^{ \mu B C }_{ A } =  P_0 + t P_1 + \sum_{ k = 1 }^n x^k P_{ k + 1 } ,
\end{align*}
for some $ P_0 , \dots , P_{ n + 1 } \in \cP $. It is easy to see that $ \hat{ Z }_a Q^{ \mu B C }_{ A } \in \cP $ for $ \hat{ Z }_a = \partial_t $ or $ \partial_{ x^j } $. On the other hand, we have
\begin{align*}
\partial_{ v^i } \left( \frac{ v^j }{ v^0 } \right) = \frac{ 1 }{ v^0 } \left( \delta_i^j - \frac{ v^i v^j }{ ( v^0 )^2 } \right) ,
\end{align*}
which shows that $ v^i \partial_{ v^j } P , v^0 \partial_{ v^i } P \in \cP $ for any $ P \in \cP $. Then, we can conclude that $ \hat{ Z }_a Q^{ \mu B C }_{ A } \in \cQ $ for $ \hat{ Z }_a = \hat{ \Omega }_{ i j } $ or $ \hat{ \Omega }_{ 0 i } $. Next, concerning the quantity $ \hat{ Z }_a \partial_{ x^\mu } ( Z_B \phi ) $, we note that
\begin{align*}
\hat{ Z }_a \partial_{ x^\mu } ( Z_B \phi ) & = \partial_{ x^\mu } \hat{ Z }_a ( Z_B \phi ) + \left[ \hat{ Z }_a , \partial_{ x^\mu } \right] ( Z_B \phi ) \\
& = \partial_{ x^\mu } ( Z_a Z_B \phi ) + c^\nu_{ a \mu } \partial_{ x^\nu } ( Z_B \phi ) ,
\end{align*}
for some constants $ c^\nu_{ a \mu } $. Hence, we may write
\begin{align*}
\hat{ Z }_a \partial_{ x^\mu } ( Z_B \phi ) = \sum_{ | B | \leq | C | \leq | B | + 1 } c^{ \nu C }_{ a \mu } \partial_{ x^\nu } ( Z_C \phi ) ,
\end{align*}
for some constants $ c^{ \nu C }_{ a \mu } $. Therefore, we conclude that
\begin{align*}
\sum_{ \substack{ | B | + | C | \leq | A | + 1 \\ 1 \leq | C | \leq | A | } } \sum_{ 0 \leq \mu \leq n } \hat{ Z }_a \left( Q^{ \mu B C }_{ A } \partial_{ x^\mu } ( Z_B \phi ) \hat{ Z }_C \right) & = \sum_{ \substack{ | B | + | C | \leq | A | + 2 \\ 1 \leq | C | \leq | A | +1 } } \sum_{ 0 \leq \mu \leq n } Q^{ \mu B C }_{ a A } \partial_{ x^\mu } ( Z_B \phi ) \hat{ Z }_C ,
\end{align*}
for some $ Q^{ \mu B C }_{ a A } \in \cQ $. This completes the proof of the lemma.
\end{proof}

Now, Lemma \ref{lem TphiZhat3} shows that the evolution equation for $ \hat{ Z }_A f $ with $ | A | \geq 1 $ can be written as follows:
\begin{align}
T_\phi \hat{ Z }_A f = \sum_{ \substack{ | B | + | C | \leq | A | + 1 \\ 1 \leq | C | \leq | A | } } \sum_{ 0 \leq \mu \leq n } Q^{ \mu B C }_{ A } \partial_{ x^\mu } ( Z_B \phi ) \hat{ Z }_C f , \label{TphiZf}
\end{align}
for some $ Q^{ \mu B C }_{ A } \in \cQ $.

\subsection{Klein--Gordon equation} \label{sec commutationphi}
It is well-known that the Klein--Gordon operator $ \Box - 1 $ commutes with all of the members in $ \bbp $. The equation for $ Z_A \phi $ can easily be derived by applying the vector fields to the both sides of the Klein--Gordon equation \eqref{KG}. Concerning the right hand side of \eqref{KG} we apply Lemma \ref{lem OmegaOmegahat} to write the equation in terms of $ Z_A \phi $ and $ \hat{ Z }_B f $.

\begin{lemma} \label{lem Zintf}
For any $ Z_A \in \bbp^{ | A | } $, we have
\begin{align*}
Z_A \int_\bbrn f \, d v = \sum_{ | B | \leq | A | } \int_\bbrn P^B \hat{ Z }_B f \, d v ,
\end{align*}
for some $ P^B \in \cP $.
\end{lemma}
\begin{proof}
The lemma is clear for $ | A | = 0 $. Since $ v^i \partial_{ v^j } P^B , v^0 \partial_{ v^i } P^B \in \cP $ as observed in the proof of Lemma \ref{lem TphiZhat3}, we obtain the desired result by iterating Lemma \ref{lem OmegaOmegahat}.
\end{proof}

Now, Lemma \ref{lem Zintf} shows that the equation for $ Z_A \phi $ can be written as
\begin{align}
( \Box - 1 ) Z_A \phi = \sum_{ | B | \leq | A | } \int_{ \bbr^n } P^B \hat{ Z }_B f \, d v , \label{ZKG}
\end{align}
for some $ P^B \in \cP $.

\section{Proof of Theorem \ref{thm}} \label{sec proof}
We are now ready to prove Theorem \ref{thm}. The proof is an application of the bootstrap argument, and we basically follow the arguments of \cite{FJS17}. First, we assume that initial data is given at $ \tau = 1 $ and the energies for the Klein--Gordon field and the distribution function are small:
\begin{align}
E_N ( \phi ) ( 1 ) + \hat{ E }_{ N, 1 } ( f ) ( 1 ) + \hat{ E }_{ N + n } ( f ) ( 1 ) < \varepsilon . \label{smalldata}
\end{align}
Then, we can find an interval $ [ 1 , T ) $, on which the corresponding solution remains small. The bootstrap assumptions for this will be given in Section \ref{sec bootsassumption}. In Section \ref{sec improvephi}, the bootstrap assumption for the Klein--Gordon field will be improved by assuming further that the distribution function satisfies a certain $ L^2 $-estimate. In Section \ref{sec improvef}, the bootstrap assumption for the distribution function will be improved. The $ L^2 $-estimate given in Section \ref{sec improvephi} for the distribution function will be verified in Section \ref{sec L2}, and we obtain the estimates \eqref{thm1}--\eqref{thm3} of Theorem \ref{thm}.

The estimates \eqref{thm41}--\eqref{thm62} of Theorem \ref{thm} are obtained by applying Propositions \ref{prop KSf} and \ref{prop KSphi} to the results \eqref{thm1}--\eqref{thm2}. For instance, we apply \eqref{KSf} of Proposition \ref{prop KSf} to \eqref{thm2} to obtain
\begin{align*}
\int_{ \bbr^n } | \hat{ Z }_A f | \, d v \lesssim \frac{ \varepsilon }{ t^{ n - 1 } \tau^{ 1 - \delta } } ,
\end{align*}
for multi-indices $ A $ satisfying $ | A | \leq N - n $. Then, since $ \tau = \sqrt{ t^2 - | x |^2 } $ and $ t > | x | $ in the future of the unit hyperboloid, we obtain the estimate \eqref{thm41}. The other estimates are obtained in a similar way, and this will complete the proof of Theorem \ref{thm}.

\subsection{Bootstrap assumption} \label{sec bootsassumption}
Suppose that initial data is given as in \eqref{smalldata}, and let $ ( \phi , f ) $ be the corresponding solution to the VKG system \eqref{V}--\eqref{KG}. For suitably regular solutions one can find an interval $ [ 1 , T ) $, on which the solution satisfies
\begin{align}
E_N ( \phi ) ( \tau ) & < 2 \varepsilon , \label{bootsphi} \\
\hat{ E }_{ N , 1 } ( f ) ( \tau ) & < 2 \varepsilon \tau^\delta , \label{bootsf}
\end{align}
where $ \varepsilon > 0 $ is the small number given in \eqref{smalldata}, and $ \delta \geq 0 $ is another small number which will be determined later. Let $ [ 1 , T ) $ be the largest interval on which the estimates \eqref{bootsphi}--\eqref{bootsf} hold.

Below, in Sections \ref{sec improvephi} and \ref{sec improvef}, we will show that the estimates \eqref{bootsphi}--\eqref{bootsf} can be improved. Then, we can conclude that the estimates \eqref{bootsphi}--\eqref{bootsf} hold for all $ 1 \leq \tau < \infty $, and this completes the proof of Theorem \ref{thm}.

\subsection{Improvement for the Klein--Gordon field} \label{sec improvephi}
In this section we improve the estimate \eqref{bootsphi}. In Lemma \ref{lem improvephi}, we assume further that higher order derivatives of the distribution function satisfy the estimate \eqref{L2} and obtain the improvement \eqref{improvephi}.

\begin{lemma} \label{lem improvephi}
Suppose that for $ N - n + 1 \leq | A | \leq N $ the distribution function $ f $ satisfies
\begin{align}
\int_{ H_\tau } \frac{ t }{ \tau } \left( \int_\bbrn | \hat{ Z }_A f | \, d v \right)^2 d \mu_{ H_\tau } \lesssim \varepsilon^2 \tau^{ 2 \delta - n } , \label{L2}
\end{align}
where $ \delta \geq 0 $ is the small number given in the assumption \eqref{bootsf}. Then we have
\begin{align}
E_N ( \phi ) ( \tau ) \leq \frac32 \varepsilon , \label{improvephi}
\end{align}
for $ 1 \leq \tau < T $.
\end{lemma}
\begin{proof}
Applying Lemma \ref{lem chiphiest} to the equation \eqref{ZKG} we obtain
\begin{align*}
& E_N ( \phi ) ( \tau ) - E_N ( \phi ) ( 1 ) \\
& \lesssim \sum_{ | A | \leq N } \sum_{ | B | \leq | A | } \int_1^\tau \int_{ H_\lambda } \left( \int_\bbrn | \hat{ Z }_B f | \, d v \right) \left| \partial_t Z_A \phi \right| \, d \mu_{ H_\lambda } d \lambda \\
& \lesssim \sum_{ | A | \leq N } \sum_{ | B | \leq | A | } \int_1^\tau \left( \int_{ H_\lambda } \frac{ t }{ \lambda } \left( \int_\bbrn | \hat{ Z }_B f | \, d v \right)^2 d \mu_{ H_\lambda } \right)^{ \frac12 } \left( \int_{ H_\lambda } \frac{ \lambda }{ t } | \partial_t Z_A \phi |^2 d \mu_{ H_\lambda } \right)^{ \frac12 } d \lambda .
\end{align*}
Since $ ( \lambda / t ) | \partial_t Z_A \phi |^2 \lesssim e ( Z_A \phi ) $ by \eqref{chiphiphi}, we have
\begin{align*}
& E_N ( \phi ) ( \tau ) - E_N ( \phi ) ( 1 ) \\
& \lesssim \sum_{ | B | \leq N } \int_1^\tau \left( \int_{ H_\lambda } \frac{ t }{ \lambda } \left( \int_\bbrn | \hat{ Z }_B f | \, d v \right)^2 d \mu_{ H_\lambda } \right)^{ \frac12 } E_N^{ \frac12 } ( \phi ) ( \lambda ) \, d \lambda \\
& \lesssim \varepsilon^\frac12 \sum_{ | B | \leq N } \int_1^\tau \left( \int_{ H_\lambda } \frac{ t }{ \lambda } \left( \int_\bbrn | \hat{ Z }_B f | \, d v \right)^2 d \mu_{ H_\lambda } \right)^\frac12 d \lambda ,
\end{align*}
by the assumption \eqref{bootsphi}. For $ | B | \leq N - n $, we apply Proposition \ref{prop KSf} to obtain
\begin{align*}
\int_{ H_\lambda } \frac{ t }{ \lambda } \left( \int_\bbrn | \hat{ Z }_B f | \, d v \right)^2 d \mu_{ H_\lambda } & \lesssim \int_{ H_\lambda } \frac{ t }{ \lambda } \left( \frac{ 1 }{ t^{ n - 1 } \lambda } \hat{ E }_{ N , 1 } ( f ) ( \lambda ) \right)^2 d \mu_{ H_\lambda } \\
& \lesssim \int_0^\infty \frac{ \varepsilon^2 \lambda^{ 2 \delta } }{ t^{ 2 n - 2 } \lambda^2 } r^{ n - 1 } d r \\
& \lesssim \varepsilon^2 \lambda^{ 2 \delta - 2 } \int_0^\infty \frac{ r^{ n - 1 } }{ ( 1 + t + r )^{ 2 n - 2 } } \, d r \\
& \lesssim \frac{ \varepsilon^2 \lambda^{ 2 \delta - 2 } }{ t^{ n - 2 } } ,
\end{align*}
where we used \eqref{dmuH_tau}. Note that the last integral converges for $ n \geq 3 $. Since $ t \geq \lambda $, we obtain
\begin{align*}
\int_{ H_\lambda } \frac{ t }{ \lambda } \left( \int_\bbrn | \hat{ Z }_B f | \, dv \right)^2 d \mu_{ H_\lambda } \lesssim \varepsilon^2 \lambda^{ 2 \delta - n } .
\end{align*}
For $ N - n + 1 \leq | B | \leq N $, we use the assumption \eqref{L2} of the lemma to obtain the same result. Hence, we obtain
\begin{align*}
& E_N ( \phi ) ( \tau ) - E_N ( \phi ) ( 1 ) \lesssim \varepsilon^{ \frac32 } \int_1^\tau \lambda^{ \delta - \frac{ n }{ 2 } } \, d \lambda ,
\end{align*}
which is integrable for small $ \delta \geq 0 $ and $ n \geq 3 $. This proves the lemma for small $ \varepsilon > 0 $.
\end{proof}

\subsection{Improvement for the distribution function} \label{sec improvef}
In this section we improve the estimate \eqref{bootsf}. Recall that the energy $ \hat{ E }_{ N , 1 } ( f ) $ defined in \eqref{fnormw} has lower order derivatives with $ v^0 $-weight and higher order derivatives without the weight. We first consider the higher order derivatives in Lemma \ref{lem intchiZf1} and then the lower order derivatives in Lemma \ref{lem intchiZf2}. In Lemma \ref{lem improvef} we combine the Lemmas \ref{lem intchiZf1} and \ref{lem intchiZf2} to obtain the desired improvement \eqref{improvef}.

\begin{lemma} \label{lem intchiZf1}
Let $ | A | \geq \lfloor N / 2 \rfloor + 1 $. Then we have
\begin{align*}
\int_{ H_\tau } \hat{ e } ( | \hat{ Z }_A f | ) \, d \mu_{ H_\tau } \leq \int_{ H_1 } \hat{ e } ( | \hat{ Z }_A f | ) \, d \mu_{ H_1 } + C
\left\{
\begin{aligned}
& \varepsilon^{ \frac32 } & ( n \geq 5 ) , \\
& \varepsilon^{ \frac32 } \delta^{ - 1 } \tau^\delta & ( n = 4 ) , 
\end{aligned}
\right.
\end{align*}
for some $ C > 0 $.
\end{lemma}
\begin{proof}
Applying Lemma \ref{lem intchif} to the equation \eqref{TphiZf} we obtain
\begin{align*}
& \int_{ H_\tau } \hat{ e } ( | \hat{ Z }_A f | ) \, d \mu_{ H_\tau } - \int_{ H_1 } \hat{ e } ( | \hat{ Z }_A f | ) \, d \mu_{ H_1 } \\
& \lesssim \sum_{ \substack{ | B | + | C | \leq | A | + 1 \\ 1 \leq | C | \leq | A | } } \sum_{ 0 \leq \mu \leq n } \int_1^\tau \int_{ H_\lambda } \int_\bbrn \left( | Q^{ \mu B C }_{ A } \partial_{ x^\mu } ( Z_B \phi ) \hat{ Z }_C f | + | \nabla_x \phi | | \hat{ Z }_A f | \right) d v \, d \mu_{ H_\lambda } d \lambda .
\end{align*}
Note that the second term in the integrand can be absorbed into the first term. Since the polynomials $ Q^{ \mu B C }_{ A } $ can be estimated as $ | Q^{ \mu B C }_{ A } | \lesssim t $, we only need to estimate
\begin{align*}
\sum_{ \substack { | B | + | C | \leq | A | + 1 \\ 1 \leq | C | \leq | A | } } \sum_{ 0 \leq \mu \leq n } \int_1^\tau \int_{ H_\lambda } t | \partial_{ x^\mu } ( Z_B \phi ) | \int_\bbrn | \hat{ Z }_C f | \, d v \, d \mu_{ H_\lambda } d \lambda .
\end{align*}
For $ | B | \leq N - \lfloor n / 2 \rfloor - 1 $, we apply Proposition \ref{prop KSphi} and Lemma \ref{lem chiintf} to obtain
\begin{align}
& \int_{ H_\lambda } t | \partial_{ x^\mu } ( Z_B \phi ) | \int_\bbrn | \hat{ Z }_C f | \, d v \,d \mu_{ H_\lambda } \nonumber \\
& \lesssim \int_{ H_\lambda } \frac{ t }{ t^{ \frac{ n }{ 2 } - 1 } \lambda } E_{ \lfloor n / 2 \rfloor +1 }^{ \frac{ 1 }{ 2 } } ( Z_B \phi ) \int_\bbrn | \hat{ Z }_C f | \, d v \, d \mu_{ H_\lambda } \nonumber \\
& \lesssim \frac{ 1 }{ \lambda^{ \frac{ n }{ 2 } - 1 } } E_N^{ \frac12 } ( \phi ) \int_{ H_\lambda } \hat{ e } ( | \hat{ Z }_C f | ) \, d \mu_{ H_\lambda } \nonumber \\
& \lesssim \frac{ 1 }{ \lambda^{ \frac{ n }{ 2 } - 1 } } E_N^{ \frac12 } ( \phi ) \hat{ E }_{ N , 1 } ( f ) . \label{intchifimprove1}
\end{align}
For $ | B | \geq N - \lfloor n / 2 \rfloor $, we notice that $ | B | + | C | \leq N + 1 $ and $ N \geq 3 n + 2 $ to conclude
\begin{align*}
| C | \leq \left\lfloor \frac{ n }{ 2 } \right\rfloor + 1 \leq \left\lfloor \frac{ N }{ 2 } \right\rfloor - n .
\end{align*}
Hence, we obtain by the estimate \eqref{KSfimprove} of Proposition \ref{prop KSf}
\begin{align*}
\int_\bbrn | \hat{ Z }_C f | \, d v & \lesssim \frac{ 1 }{ t^n } \sum_{ | D | \leq n } \int_{ H_\lambda } \hat{ e } ( v^0 | \hat{ Z }_D \hat{ Z }_C f | ) \, d \mu_{ H_\lambda } \\
& \lesssim \frac{ 1 }{ t^n } \hat{ E }_{ N , 1 } ( f ) ,
\end{align*}
which shows that
\begin{align}
& \int_{ H_\lambda } t | \partial_{ x^\mu } ( Z_B \phi ) | \int_\bbrn | \hat{ Z }_C f | \, d v \, d \mu_{ H_\lambda } \nonumber \\
& \lesssim \hat{ E }_{ N , 1 } ( f ) \int_{ H_\lambda } | \partial_{ x^\mu } ( Z_B \phi ) | \frac{ 1 }{ t^{ n - 1 } } d \mu_{ H_\lambda } \nonumber \\
& \lesssim \hat{ E }_{ N , 1 } ( f ) \left( \int_{ H_\lambda } | \partial_{ x^\mu } ( Z_B \phi ) |^2 \frac{ \lambda }{ t } \, d \mu_{ H_\lambda } \right)^{ \frac{ 1 }{ 2 } } \left( \int_{ H_\lambda } \frac{ 1 }{ t^{ 2 n - 2 } } \frac{ t }{ \lambda } \, d \mu_{ H_\lambda } \right)^{ \frac{ 1 }{ 2 } } \nonumber \\
& \lesssim \hat{ E }_{ N , 1 } ( f ) \left( \int_{ H_\lambda } e ( Z_B \phi ) \, d \mu_{ H_\lambda } \right)^{ \frac{ 1 }{ 2 } } \left( \int_0^\infty \frac{ r^{ n - 1 } }{ t^{ 2 n - 2 } } \, d r \right)^{ \frac{ 1 }{ 2 } } \nonumber \\
& \lesssim \hat{ E }_{ N , 1 } ( f ) E_N^{ \frac12 } ( \phi ) \left( \frac{ 1 }{ t^{ n - 2 } } \right)^{ \frac12 } , \label{intchifimprove2}
\end{align}
where we used \eqref{chiphiphi} and \eqref{dmuH_tau}, and the last inequality holds for $ n \geq 3 $. We combine \eqref{intchifimprove1} and \eqref{intchifimprove2} to conclude that
\begin{align*}
\int_{ H_\tau } \hat{ e } ( | \hat{ Z }_A f | ) \, d \mu_{ H_\tau } - \int_{ H_1 } \hat{ e } ( | \hat{ Z }_A f | ) \, d \mu_{ H_1 } & \lesssim \varepsilon^{ \frac32 } \int_1^\tau \lambda^{ \delta + 1 - \frac{ n }{ 2 } } d \lambda \\
& \lesssim \frac{ \varepsilon^{ \frac32 } ( \tau^{ \delta + 2 - \frac{ n }{ 2 } } - 1 ) }{ \delta + 2 - \frac{ n }{ 2 } } ,
\end{align*}
which completes the proof of the lemma.
\end{proof}

\begin{lemma} \label{lem intchiZf2}
Let $ | A | \leq \lfloor N / 2 \rfloor $. Then we have
\begin{align*}
\int_{ H_\tau } \hat{ e } ( v^0 | \hat{ Z }_A f | ) \, d \mu_{ H_\tau } \leq \int_{ H_1 } \hat{ e } ( v^0 | \hat{ Z }_A f | ) \, d \mu_{ H_1 } + C
\left\{
\begin{aligned}
& \varepsilon^{ \frac32 } & ( n \geq 5 ) , \\
& \varepsilon^{ \frac32 } \delta^{ - 1 } \tau^\delta & ( n = 4 ) , \\
\end{aligned}
\right.
\end{align*}
for some $ C > 0 $.
\end{lemma}
\begin{proof}
Let us consider the evolution equation of $ v^0 \hat{ Z }_A f $:
\begin{align}
T_{ \phi } \left[ v^0 \hat{ Z }_A f \right] = ( T_\phi v^0 ) \hat{ Z }_A f + v^0 T_\phi \hat{ Z }_A f . \label{TphivZf}
\end{align}
For the first term on the right hand side we notice that $ T_\phi v^0 = - v \cdot \nabla_x \phi $. We use Proposition \ref{prop KSphi} to obtain
\begin{align*}
| T_\phi v^0 | \lesssim \frac{ \varepsilon^{ \frac12 } }{ t^{ \frac{ n }{ 2 } - 1 } \tau } v^0 .
\end{align*}
For the second term we note that the quantity $ T_\phi \hat{ Z }_A f $ can be written as in \eqref{TphiZf} for $ | A | \geq 1 $. We now apply Lemma \ref{lem intchif} to the equation \eqref{TphivZf} to obtain the following estimate:
\begin{align*}
& \int_{ H_\tau } \hat{ e } ( v^0 | \hat{ Z }_A f | ) \, d \mu_{ H_\tau } - \int_{ H_1 } \hat{ e } ( v^0 | \hat{ Z }_A f | ) \, d \mu_{ H_1 } \\
& \lesssim \int_1^\tau \int_{ H_\lambda } \int_\bbrn \left| ( T_\phi v^0 ) \hat{ Z }_A f + v^0 T_\phi \hat{ Z }_A f \right| + | \nabla_x \phi | | v^0 \hat{ Z }_A f | \, d v \, d \mu_{ H_\lambda } d \lambda \\
& \lesssim \int_1^\tau \int_{ H_\lambda } \int_\bbrn \frac{ \varepsilon^{ \frac12 } }{ t^{ \frac{ n }{ 2 } - 1 } \lambda } v^0 | \hat{ Z }_A f | \, d v \, d \mu_{ H_\lambda } d \lambda \\
& \quad + \sum_{ \substack{ | B | + | C | \leq | A | + 1 \\ 1 \leq | C | \leq | A | } } \sum_{ 0 \leq \mu \leq n } \int_1^\tau \int_{ H_\lambda } \int_\bbrn v^0 | Q^{ \mu B C }_{ A } | | \partial_{ x^\mu } ( Z_B \phi ) | | \hat{ Z }_C f | \, d v \, d \mu_{ H_\lambda } d \lambda ,
\end{align*}
for some $ Q^{ \mu B C }_{ A } \in \cQ $. The first quantity is estimated by using Lemma \ref{lem chiintf} as
\begin{align*}
\int_{ H_\lambda } \int_\bbrn \frac{ \varepsilon^{ \frac12 } }{ t^{ \frac{ n }{ 2 } - 1 } \lambda } v^0 | \hat{ Z }_A f | \, d v \, d \mu_{ H_\lambda } & \lesssim \frac{ \varepsilon^{ \frac12 } }{ \lambda^{ \frac{ n }{ 2 } } } \int_{ H_\lambda } \hat{ e } ( v^0 | \hat{ Z }_A f | ) \, d \mu_{ H_\lambda } \\
& \lesssim \frac{ \varepsilon^{ \frac12 } }{ \lambda^{ \frac{ n }{ 2 } } } \hat{ E }_{ N , 1 } ( f ) ,
\end{align*}
since $ | A | \leq \lfloor N / 2 \rfloor $. For the second quantity we notice that $ | B | \leq \lfloor N / 2 \rfloor $ and
\begin{align*}
\left\lfloor \frac{ N }{ 2 } \right\rfloor + \left\lfloor \frac{ n }{ 2 } \right\rfloor + 1 \leq \left\lfloor \frac{ N + n + 2 }{ 2 } \right\rfloor \leq N ,
\end{align*}
since $ N \geq 3 n + 2 $. Now, we apply Proposition \ref{prop KSphi} to obtain
\begin{align*}
& \int_{ H_\lambda } \int_\bbrn v^0 | Q^{ \mu B C }_{ A } | | \partial_{ x^\mu } ( Z_B \phi ) | | \hat{ Z }_C f | \, d v \, d \mu_{ H_\lambda } \\
& \lesssim \int_{ H_\lambda } \frac{ t }{ t^{ \frac{ n }{ 2 } - 1 } \lambda } E_{ \lfloor n / 2 \rfloor + 1 }^{ \frac12 } ( Z_B \phi ) \int_\bbrn v^0 | \hat{ Z }_C f | \, d v \, d \mu_{ H_\lambda } \\
& \lesssim \frac{ 1 }{ \lambda^{ \frac{ n }{ 2 } - 1 } } E_N^{ \frac12 } ( \phi ) \int_{ H_\lambda } \hat{ e } ( v^0 | \hat{ Z }_C f | ) \, d \mu_{ H_\lambda } \\
& \lesssim \varepsilon^{ \frac{ 3 }{ 2 } } \lambda^{ \delta + 1 - \frac{ n }{ 2 } } .
\end{align*}
Integrating the above quantities for $ 1 \leq \lambda \leq \tau $ as in the previous lemma we obtain the desired result for $ | A | \geq 1 $. For $ | A | = 0 $, we can easily obtain the same result, since the second term on the right hand side of \eqref{TphivZf} vanishes, and this completes the proof of the lemma.
\end{proof}

\begin{lemma} \label{lem improvef}
For $ 1 \leq \tau < T $, we have 
\begin{align}
\hat{ E }_{ N , 1 } ( f ) ( \tau ) \leq \frac{ 3 }{ 2 } \varepsilon \tau^\delta , \label{improvef}
\end{align}
where $ \delta = 0 $ for $ n \geq 5 $, and $ \delta = \varepsilon^{ 1 / 4 } $ for $ n = 4 $. 
\end{lemma}
\begin{proof}
We combine Lemmas \ref{lem intchiZf1} and \ref{lem intchiZf2} to obtain for $ n \geq 5 $,
\begin{align*}
\hat{ E }_{ N , 1 } ( f ) ( \tau ) \leq \hat{ E }_{ N , 1 } ( f ) ( 1 ) + C \varepsilon^{ \frac32 } .
\end{align*}
Since $ \hat{ E }_{ N , 1 } ( f ) ( 1 ) < \varepsilon $ and $ \varepsilon > 0 $ is small, we obtain the desired result for $ n \geq 5 $. For $ n = 4 $, we choose $ \delta = \varepsilon^{ 1 / 4 } $ to obtain
\begin{align*}
\hat{ E }_{ N , 1 } ( f ) ( \tau ) \leq \hat{ E }_{ N , 1 } ( f ) ( 1 ) + C \varepsilon^{ \frac54 } \tau^\delta .
\end{align*}
Since $ \varepsilon > 0 $ is small, we obtain the desired result.
\end{proof}

\subsection{$ L^2 $-estimate for the distribution function} \label{sec L2}
In order to complete the bootstrap arguments we need to show that the $ L^2 $-estimate \eqref{L2} for the higher order derivatives $ \hat{ Z }_A f $ of the distribution function can be derived from the bootstrap assumptions \eqref{bootsphi} and \eqref{bootsf}. Recall that the evolution equation of $ \hat{ Z }_A f $ with $ | A | \geq 1 $ is given by \eqref{TphiZf}, which can be written as
\begin{align}
T_\phi \hat{ Z }_A f & = \sum_{ 1 \leq | C | \leq | A | } R^C_A \hat{ Z }_C f , \label{TphiZf'}
\end{align}
where the coefficients are given by
\begin{align}
R^C_A = \sum_{ | B | + | C | \leq | A | + 1 } \sum_{ 0 \leq \mu \leq n } Q^{ \mu B C }_{ A } \partial_{ x^\mu } ( Z_B \phi ) . \label{coefficient}
\end{align}
We can estimate the coefficients as follows.

\begin{lemma} \label{lem R}
The coefficient functions $ R^C_A $ given by \eqref{coefficient} satisfy the following estimates:
\begin{enumerate}
\item[$ ( i ) $] For $ \lfloor N / 2 \rfloor - n + 1 \leq | A | \leq N $ and $ \lfloor N / 2 \rfloor - n + 1 \leq | C | \leq | A | $, we have
\begin{align*}
| R^C_A | \lesssim \frac{ \varepsilon^\frac{ 1 }{2 } }{ t^{ \frac{ n }{ 2 } - 2 } \tau } .
\end{align*}
\item[$ ( ii ) $] For $ \lfloor N / 2 \rfloor - n + 1 \leq | A | \leq N $ and $ 1 \leq | C | \leq \lfloor N / 2 \rfloor - n $, we have
\begin{align*}
\int_{ H_\tau } \frac{ \tau }{ t^3 } | R^C_A |^2 d \mu_{ H_\tau } \lesssim \varepsilon .
\end{align*}
\item[$ ( iii ) $] For $ 1 \leq | A | \leq \lfloor N / 2 \rfloor - n $ and $ 1 \leq | C | \leq | A | $, we have
\begin{align*}
| R^C_A | \lesssim \frac{ \varepsilon^\frac{ 1 }{ 2 } }{ t^{ \frac{ n }{ 2 } - 2 } \tau } .
\end{align*}
\end{enumerate}
\end{lemma}
\begin{proof}
For $ ( i ) $ we use Proposition \ref{prop KSphi} to obtain
\begin{align}
| \partial_{ x^\mu } ( Z_B \phi ) | & \lesssim \frac{ 1 }{ t^{ \frac{ n }{ 2 } - 1 } \tau } E_{ \lfloor n / 2 \rfloor + 1 }^{ \frac12 } ( Z_B \phi ) , \label{coefficientQ}
\end{align}
and notice that $ | B | $ is small in the sense that
\begin{align*}
\left\lfloor \frac{ n }{ 2 } \right\rfloor +1 + | B | & \leq \left\lfloor \frac{ n }{ 2 } \right\rfloor + 1 + | A | + 1 - | C | \\
& \leq \left\lfloor \frac{ n }{ 2 } \right\rfloor + 1 + N - \left\lfloor \frac{ N }{ 2 } \right\rfloor + n \\
& \leq \left\lfloor \frac{ 3 n + 2 }{ 2 } \right\rfloor + N - \left\lfloor \frac{ N }{ 2 } \right\rfloor \\
& \leq N ,
\end{align*}
since $ N \geq 3 n + 2 $. Hence, we can apply the bootstrap assumption \eqref{bootsphi} to \eqref{coefficientQ} to obtain
\begin{align*}
| \partial_{ x^\mu } ( Z_B \phi ) | & \lesssim \frac{ \varepsilon^\frac{ 1 }{ 2 } }{ t^{ \frac{ n }{ 2 } - 1 } \tau } .
\end{align*}
Since $ | Q^{ \mu B C }_{ A } | \lesssim t $, we obtain the desired result for $ ( i ) $. For $ ( i i ) $ we can only use \eqref{chiphiphi} to have
\begin{align*}
| \partial_{ x^\mu } ( Z_B \phi ) |^2 \lesssim \frac{ t }{ \tau } e ( Z_B \phi ) .
\end{align*}
Applying this to \eqref{coefficient} we obtain
\begin{align*}
\int_{ H_\tau } \frac{ \tau }{ t^3 } | R^C_A |^2 d \mu_{ H_\tau } & \lesssim \sum_{ | B | + | C | \leq | A | + 1 } \sum_{ 0 \leq \mu \leq n } \int_{ H_\tau } \frac{ \tau }{ t^3 } | Q^{ \mu B C }_{ A } |^2 | \partial_{ x^\mu } ( Z_B \phi ) |^2 d \mu_{ H_\tau } \\
& \lesssim \sum_{ | B | \leq N } \int_{ H_\tau } e ( Z_B \phi ) \, d \mu_{ H_\tau } \\
& \lesssim \varepsilon ,
\end{align*}
since $ | Q^{ \mu B C }_{ A } | \lesssim t $. For $ ( iii ) $ we have
\begin{align*}
\left\lfloor \frac{ n }{ 2 } \right\rfloor +1 + | B | \leq \left\lfloor \frac{ n }{ 2 } \right\rfloor +1 + | A | \leq \left\lfloor \frac{ n }{ 2 } \right\rfloor + 1 + \left\lfloor \frac{ N }{ 2 } \right\rfloor - n \leq N ,
\end{align*}
since $ N \geq 3 n + 2 $. Hence, we obtain the same result as in the case of $ ( i ) $.
\end{proof}

Now, we use Lemma \ref{lem R} to rewrite the equation \eqref{TphiZf'} as in \cite{FJS17}. Let $ F $ and $ F_0 $ denote the vector-valued functions given by
\begin{align*}
F : = \left( \hat{ Z }_{ A_1 } f , \dots , \hat{ Z }_{ A_q } f \right) , \qquad F_0 : = \left( \hat{ Z }_{ B_{ 1 } } f , \dots , \hat{ Z }_{ B_r } f \right) ,
\end{align*}
where $ A_i $ and $ B_j $ are all the possible multi-indices satisfying
\begin{align*}
\left\lfloor \frac{ N }{ 2 } \right\rfloor - n + 1 \leq | A_i | \leq N , \qquad 1 \leq | B_j | \leq \left\lfloor \frac{ N }{ 2 } \right\rfloor - n ,
\end{align*}
with suitable orderings. Then, the equation \eqref{TphiZf'} can be decomposed into the equations for $ F $ and $ F_0 $ given by
\begin{align}
T_\phi F + R_1 F & = R_2 F_0 , \label{TphiF} \\
T_\phi F_0 & = R_3 F_0 , \label{TphiL}
\end{align}
where $ R_1 $, $ R_2 $ and $ R_3 $ are matrix-valued functions satisfying
\begin{align}
\sup_v \| R_1 \| \lesssim \frac{ \varepsilon^{ \frac{ 1 }{ 2 } } }{ t^{ \frac{ n }{ 2 } - 2 } \tau } , \qquad \int_{ H_\tau } \frac{ \tau }{ t^3 } \sup_v \| R_2 \|^2 \, d \mu_{ H_\tau } \lesssim \varepsilon , \qquad \sup_v \| R_3 \| \lesssim \frac{ \varepsilon^{ \frac{ 1 }{ 2 } } }{ t^{ \frac{ n }{ 2 } - 2 } \tau } . \label{AAB}
\end{align}
Here, $ \| \cdot \| $ denotes a suitable matrix norm, for instance $ \| R_1 \| = \left( \sum_{ a , b } ( ( R_1 )^a_b )^2 \right)^{ 1 / 2 } $. Below, we will show that $ F $ satisfies the $ L^2 $-estimate \eqref{L2}.

\subsubsection{Decomposition of solution}
We note that $ F $ satisfies the equation \eqref{TphiF}, where the coefficients $ R_1 $ in the homogeneous part and $ R_2 $ in the inhomogeneous part have different behavior as in \eqref{AAB}. We decompose $ F $ into the homogeneous and inhomogeneous parts as follows:
\begin{align*}
F = F_1 + F_2 ,
\end{align*}
where $ F_1 $ is the solution of the following homogeneous equation:
\begin{align} \label{TphiH}
T_\phi F_1 + R_1 F_1 = 0 ,
\end{align}
with initial data $ F_1 ( 1 ) = F ( 1 ) $, and $ F_2 $ is the solution of the following inhomogeneous equation:
\begin{align} \label{TphiI}
T_\phi F_2 + R_1 F_2 = R_2 F_0 ,
\end{align}
with vanishing data $ F_2 ( 1 ) = 0 $.

\subsubsection{Homogeneous part}
We recall that $ F $ has the higher order derivatives of $ f $ of order up to $ N $. If we want to apply Proposition \ref{prop KSf} to $ F $, then we need to estimate the derivatives of $ F $ of order up to $ n $, which requires the derivatives of $ f $ of order up to $ N + n $. The following lemma shows that if we only consider the homogeneous part $ F_1 $, then we can simply assume that the derivatives of $ f $ of order up to $ N + n $ are small at $ \tau = 1 $ to estimate the derivatives of $ F_1 $ of order up to $ n $.

\begin{lemma} \label{lem homo}
The homogeneous part $ F_1 $ satisfies
\begin{align}
\hat{ E }_n ( F_1 ) ( \tau ) \lesssim \varepsilon \tau^\delta , \label{homo}
\end{align}
where $ \delta = 0 $ for $ n \geq 5 $, and $ \delta = \varepsilon^{ 1 / 4 } $ for $ n = 4 $.
\end{lemma}
\begin{proof}
Applying $ \hat{ Z }_D $ to the equation \eqref{TphiH} we obtain
\begin{align}
T_\phi \hat{ Z }_D F_1 = - \hat{ Z }_D \left( R_1 F_1 \right) + \left[ T_\phi , \hat{ Z }_D \right] F_1 . \label{TphiZF}
\end{align}
Concerning the first term on the right hand side we notice that $ R_1 $ is given by
\begin{align*}
R_1 = \sum_{ | B | + | C | \leq | A | + 1 } \sum_{ 0 \leq \mu \leq n } Q^{ \mu B C }_{ A } \partial_{ x^\mu } ( Z_B \phi ) ,
\end{align*}
with multi-indices $ A $ and $ C $ satisfying
\begin{align*}
\left\lfloor \frac{ N }{ 2 } \right\rfloor - n +1 \leq | A | \leq N , \qquad \left\lfloor \frac{ N }{ 2 } \right\rfloor - n +1 \leq | C | \leq | A | .
\end{align*}
Then, we obtain for any $ | D | \leq n $,
\begin{align*}
| \partial_{ x^\mu } ( Z_{ D } Z_B \phi ) | & \lesssim \frac{ 1 }{ t^{ \frac{ n }{ 2 } - 1 } \tau } E_{ \lfloor n / 2 \rfloor + 1 }^{ \frac12 } ( Z_{ D } Z_B \phi ) \\
& \lesssim \frac{ \varepsilon^{ \frac{ 1 }{ 2 } } }{ t^{ \frac{ n }{ 2 } - 1 } \tau } ,
\end{align*}
by Proposition \ref{prop KSphi} and the bootstrap assumption \eqref{bootsphi}, since $ N \geq 5 n + 2 $ and
\begin{align*}
\left\lfloor \frac{ n }{ 2 } \right\rfloor + 1 + | D | + | B | & \leq \left\lfloor \frac{ n }{ 2 } \right\rfloor + 1 + n + N + 1 - \left\lfloor \frac{ N }{ 2 } \right\rfloor + n - 1 \\
& \leq \left\lfloor \frac{ 5 n + 2 }{ 2 } \right\rfloor + N - \left\lfloor \frac{ N }{ 2 } \right\rfloor \\
& \leq N .
\end{align*}
We recall that $ \hat{ Z }_{ D } Q^{ \mu B C }_{ A } \in \cQ $, so that $ | \hat{ Z }_{ D } Q^{ \mu B C }_{ A } | \lesssim t $, and note that $ Z_{ D } $ commutes with $ \partial_{ x^\mu } $ as in Lemma \ref{lem basic}. Now, we can conclude that for any $ | D | \leq n $,
\begin{align*}
| \hat{ Z }_D ( R_1 F_1 ) | & \lesssim \frac{ \varepsilon^{ \frac{ 1 }{ 2 } } }{ t^{ \frac{ n }{ 2 } - 2 } \tau } \sum_{ | D_1 | \leq n } | \hat{ Z }_{ D_1 } F_1 | .
\end{align*}
For the second term on the right hand side of \eqref{TphiZF} we use Lemma \ref{lem TphiZhat3} to have
\begin{align*}
\left[ T_\phi , \hat{ Z }_D \right] = \sum_{ \substack{ | B | + | C | \leq | D | + 1 \\ 1 \leq | C | \leq | D | } } \sum_{ 0 \leq \mu \leq n } Q^{ \mu B C }_{ D } \partial_{ x^\mu } ( Z_B \phi ) \hat{ Z }_C ,
\end{align*}
for some $ Q^{ \mu B C }_{ D } \in \cQ $. Since $ | D | \leq n $, we obtain by the same arguments as above
\begin{align*}
\left| \left[ T_\phi , \hat{ Z }_D \right] F_1 \right| & \lesssim \frac{ \varepsilon^{ \frac{ 1 }{ 2 } } }{ t^{ \frac{ n }{ 2 } - 2 } \tau } \sum_{ | D_1 | \leq n } | \hat{ Z }_{ D_1 } F_1 | .
\end{align*}
Now, we apply Lemma \ref{lem intchif} to obtain
\begin{align*}
& \int_{ H_\tau } \hat{ e } ( | \hat{ Z }_D F_1 | ) \, d \mu_{ H_\tau } - \int_{ H_1 } \hat{ e } ( | \hat{ Z }_D F_1 | ) \, d \mu_{ H_1 } \\
& \lesssim \int_1^\tau \int_{ H_\lambda } \int_\bbrn | \hat{ Z }_D ( R_1 F_1 ) | + \left| \left[ T_\phi , \hat{ Z }_D \right] F_1 \right| + | \nabla_x \phi | | \hat{ Z }_D F_1 | \, d v \, d \mu_{ H_\lambda } d \lambda \\
& \lesssim \int_1^\tau \int_{ H_\lambda } \frac{ \varepsilon^{ \frac{ 1 }{ 2 } } }{ t^{ \frac{ n }{ 2 } - 2 } \lambda } \sum_{ | D_1 | \leq n } \hat{ e } ( | \hat{ Z }_{ D_1 } F_1 | ) \, d \mu_{ H_\lambda } d \lambda \\
& \lesssim \varepsilon^{ \frac{ 1 }{ 2 } } \int_1^\tau \frac{ 1 }{ \lambda^{ \frac{ n }{ 2 } - 1 } } \hat{ E }_n ( F_1 ) \, d \lambda ,
\end{align*}
which implies that
\begin{align}
\hat{ E }_n ( F_1 ) ( \tau ) \leq \hat{ E }_n ( F_1 ) ( 1 ) + C \varepsilon^{ \frac{ 1 }{ 2 } } \int_1^\tau \frac{ 1 }{ \lambda^{ \frac{ n }{ 2 } - 1 } } \hat{ E }_n ( F_1 ) ( \lambda ) \, d \lambda . \label{FGronwall}
\end{align}
Since $ F_1 ( 1 ) = F ( 1 ) $ and $ \hat{ E }_n ( F ) ( 1 ) \lesssim \hat{ E }_{ N + n } ( f ) ( 1 ) < \varepsilon $, we can find $ C_{ 1 } > 0 $ and $ \tau_{ 1 } > 1 $ such that
\begin{align}
\hat{ E }_n ( F_1 ) ( 1 ) \leq C_{ 1 } \varepsilon , \qquad \hat{ E }_n ( F_1 ) ( \tau ) \leq
\left\{
\begin{aligned}
& 2 C_{ 1 } \varepsilon & ( n \geq 5 ) , \\
& 2 C_{ 1 } \varepsilon \tau^\delta & ( n = 4 ) ,
\end{aligned}
\right. \label{Flocal}
\end{align}
for $ 1 \leq \tau < \tau_{ 1 } $. Then, the inequality \eqref{FGronwall} implies that
\begin{align*}
\hat{ E }_n ( F_1 ) ( \tau ) \leq C_{ 1 } \varepsilon + C
\left\{
\begin{aligned}
& C_{ 1 } \varepsilon^{ \frac{ 3 }{ 2 } } & ( n \geq 5 ) , \\
& C_{ 1 } \varepsilon^{ \frac{ 3 }{ 2 } } \delta^{ - 1 } \tau^\delta & ( n = 4 ) ,
\end{aligned}
\right.
\end{align*}
for $ 1 \leq \tau < \tau_{ 1 } $. We take $ \delta = \varepsilon^{ 1 / 4 } $ to conclude that \eqref{Flocal} holds for all $ \tau \geq 1 $.
\end{proof}

\subsubsection{Inhomogeneous part}
To estimate the inhomogeneous part $ F_2 $ we introduce the matrix-valued function $ K $ satisfying
\begin{align}
T_\phi K + R_1 K + K R_3 = R_2 , \label{TphiK}
\end{align}
with vanishing data $ K ( 1 ) = 0 $. We notice that $ K $ is a $ q \times r $ matrix and $ F_2 $ is given by
\begin{align*}
F_2 = K F_0 .
\end{align*}
Below, we will need to estimate the following quantity:
\begin{align*}
\hat{ E }_K : = \int_{ H_\tau } \hat{ e } ( \| K \|^2 | F_0 | ) \, d \mu_{ H_\tau } ,
\end{align*}
which is the same as $ \hat{ E }_0 ( \| K \|^2 | F_0 | ) $ by the definition of \eqref{fnorm}, but we will use the notation $ \hat{ E }_K $ for simplicity. In the following lemma, we will estimate the quantity $ \hat{ E }_K $ by using the properties of $ R_1 $, $ R_2 $ and $ R_3 $ of \eqref{AAB} in a suitable way.

\begin{lemma} \label{lem inhomo}
The quantity $ \hat{ E }_K $ satisfies
\begin{align*}
\hat{ E }_K ( \tau ) \lesssim \varepsilon \tau^\delta ,
\end{align*}
where $ \delta = 0 $ for $ n \geq 5 $, and $ \delta = \varepsilon^{ 1 / 4 } $ for $ n = 4 $.
\end{lemma}
\begin{proof}
The proof of this lemma is basically the same as the proof of Lemma \ref{lem homo}. For each $ a $ and $ b $, we consider the evolution equation of $ ( K^a_b )^2 F_0 $:
\begin{align*}
T_\phi ( ( K^a_b )^2 F_0 ) = - 2 ( R_1 K )^a_b K^a_b F_0 - 2 ( K R_3 )^a_b K^a_b F_0 + 2 ( R_2 )^a_b K^a_b F_0 + ( K^a_b )^2 R_3 F_0 ,
\end{align*}
by the equations \eqref{TphiK} and \eqref{TphiL}. Here, the repeated indices are not summed over. In order to apply Lemma \ref{lem intchif}, we need to estimate the right hand side. For the terms having $ R_1 $ or $ R_3 $, we apply \eqref{AAB} to obtain
\begin{align*}
& \int_{ H_\tau } \int_\bbrn 2 | ( R_1 K )^a_b K^a_b F_0 | + 2 | ( K R_3 )^a_b K^a_b F_0 | + ( K^a_b )^2 | R_3 F_0 | \, d v \, d \mu_{ H_\tau } \\
& \lesssim \int_{ H_\tau } \frac{ \varepsilon^\frac{ 1 }{ 2 } }{ t^{ \frac{ n }{ 2 } - 2 } \tau } \int_\bbrn \| K \|^2 | F_0 | \, d v \, d \mu_{ H_\tau } \\
& \lesssim \int_{ H_\tau } \frac{ \varepsilon^\frac{ 1 }{ 2 } }{ t^{ \frac{ n }{ 2 } - 2 } \tau } \hat{ e } ( \| K \|^2 | F_0 | ) \, d \mu_{ H_\tau } \\
& \lesssim \frac{ \varepsilon^\frac{ 1 }{ 2 } }{ \tau^{ \frac{ n }{ 2 } - 1 } } \hat{ E }_K .
\end{align*}
For the terms having $ R_2 $, we use the Cauchy-Schwarz inequality to obtain
\begin{align*}
& \int_{ H_\tau } \int_\bbrn 2 | ( R_2 )^a_b K^a_b F_0 | \, d v \, d \mu_{ H_\tau } \\
& \lesssim \left( \int_{ H_\tau } \int_\bbrn ( ( R_2 )^a_b )^2 | F_0 | \, d v \, d \mu_{ H_\tau } \right)^{ \frac{ 1 }{ 2 } } \left( \int_{ H_\tau } \int_\bbrn ( K^a_b )^2 | F_0 | \, d v \, d \mu_{ H_\tau } \right)^{ \frac{ 1 }{ 2 } } .
\end{align*}
The first integration on the right hand side is estimated by using \eqref{AAB},
\begin{align*}
\left( \int_{ H_\tau } \int_\bbrn ( ( R_2 )^a_b )^2 | F_0 | \, d v \, d \mu_{ H_\tau } \right)^{ \frac{ 1 }{ 2 } }
& \lesssim \left( \int_{ H_\tau } \sup_v \| R_2 \|^2 \int_\bbrn | F_0 | \, d v \, d \mu_{ H_\tau } \right)^{ \frac{ 1 }{ 2 } } \\
& \lesssim \left( \int_{ H_\tau } \sup_v \| R_2 \|^2 \frac{ 1 }{ t^{ n - 1 } \tau } \hat{ E }_n ( F_0 ) \, d \mu_{ H_\tau } \right)^{ \frac{ 1 }{ 2 } } \\
& \lesssim \left( \int_{ H_\tau } \frac{ 1 }{ t^{ n - 1 } \tau } \sup_v \| R_2 \|^2 \, d \mu_{ H_\tau } \right)^{ \frac{ 1 }{ 2 } } \hat{ E }_n^{ \frac{ 1 }{ 2 } } ( F_0 ) \\
& \lesssim \frac{ \varepsilon^{ \frac{ 1 }{ 2 } } }{ \tau^{ \frac{ n }{ 2 } - 1 } } \hat{ E }_N^{ \frac{ 1 }{ 2 } } ( f ) ,
\end{align*}
and the second integration is simply estimated as follows:
\begin{align*}
\left( \int_{ H_\tau } \int_\bbrn ( K^a_b )^2 | F_0 | \, d v \, d \mu_{ H_\tau } \right)^{ \frac{ 1 }{ 2 } } & \lesssim \hat{ E }^{ \frac{ 1 }{ 2 } }_K .
\end{align*}
Now, we apply Lemma \ref{lem intchif} and consider all the possible indices $ a $ and $ b $ to obtain
\begin{align}
& \hat{ E }_K ( \tau ) \lesssim \int_1^\tau \frac{ \varepsilon^{ \frac{ 1 }{ 2 } } }{ \lambda^{ \frac{ n }{ 2 } - 1 } } \hat{ E }_K ( \lambda ) + \frac{ \varepsilon^{ \frac{ 1 }{ 2 } } }{ \lambda^{ \frac{ n }{ 2 } - 1 } } \hat{ E }_N^{ \frac{ 1 }{ 2 } } ( f ) ( \lambda ) \hat{ E }^{ \frac{ 1 }{ 2 } }_K ( \lambda ) \, d \lambda , \label{EGGronwall}
\end{align}
since $ K ( 1 ) = 0 $. We can find $ C_2 > 0 $ and $ \tau_2 > 1 $ satisfying
\begin{align}
\hat{ E }_K ( \tau ) \leq
\left\{
\begin{aligned}
& C_2 \varepsilon & ( n \geq 5 ) , \\
& C_2 \varepsilon \tau^\delta & ( n = 4 ) ,
\end{aligned}
\right. \label{EGlocal}
\end{align}
for $ 1 \leq \tau < \tau_2 $. Then, the inequality \eqref{EGGronwall} together with Lemma \ref{lem improvef} implies that
\begin{align*}
\hat{ E }_K ( \tau ) \lesssim
\left\{
\begin{aligned}
& \int_1^\tau \frac{ \varepsilon^{ \frac{ 3 }{ 2 } } }{ \lambda^{ \frac{ n }{ 2 } - 1 } } \, d \lambda \lesssim \varepsilon^{ \frac{ 3 }{ 2 } } & ( n \geq 5 ) , \\
& \int_1^\tau \frac{ \varepsilon^{ \frac{ 3 }{ 2 } } }{ \lambda^{ \frac{ n }{ 2 } - 1 - \delta } } \, d \lambda \lesssim \varepsilon^{ \frac{ 3 }{ 2 } } \delta^{ - 1 } \tau^\delta & ( n = 4 ) ,
\end{aligned}
\right.
\end{align*}
for $ 1 \leq \tau \leq \tau_2 $. We now take $ \delta = \varepsilon^{ 1 / 4 } $ to conclude that \eqref{EGlocal} holds for all $ \tau \geq 1 $.
\end{proof}

\subsubsection{Proof of the $ L^2 $-estimate}
We combine Lemmas \ref{lem homo} and \ref{lem inhomo} to prove the $ L^2 $-estimate \eqref{L2} for the distribution function. We note that Lemma \ref{lem homo} together with Proposition \ref{prop KSf} implies
\begin{align}
\int_\bbrn | F_1 | \, d v \lesssim \frac{ \varepsilon }{ t^{ n - 1 } \tau^{ 1 - \delta } } , \label{Fest}
\end{align}
where $ \delta = 0 $ for $ n \geq 5 $, and $ \delta = \varepsilon^{ 1 / 4 } $ for $ n = 4 $. Now, we obtain the $ L^2 $-estimate in the following proposition.

\begin{proposition}
For $ N - n + 1 \leq | A | \leq N $, the distribution function satisfies
\begin{align*}
\int_{ H_\tau } \frac{ t }{ \tau } \left( \int_\bbrn | \hat{ Z }_A f | \, d v \right)^2 d \mu_{ H_\tau } \lesssim \varepsilon^2 \tau^{ 2 \delta - n } ,
\end{align*}
where $ \delta = 0 $ for $ n \geq 5 $, and $ \delta = \varepsilon^{ 1 / 4 } $ for $ n = 4 $.
\end{proposition}
\begin{proof}
We notice that it is enough to show that
\begin{align*}
\int_{ H_\tau } \frac{ t }{ \tau } \left( \int_\bbrn | F | \, d v \right)^2 d \mu_{ H_\tau } \lesssim \varepsilon^2 \tau^{ 2 \delta - n } ,
\end{align*}
and we consider $ F = F_1 + F_2 $ separately. The homogeneous part $ F_1 $ is estimated by using \eqref{Fest} with \eqref{dmuH_tau} as follows:
\begin{align*}
\int_{ H_\tau } \frac{ t }{ \tau } \left( \int_\bbrn | F_1 | \, d v \right)^2 d \mu_{ H_\tau } & \lesssim \int_{ H_\tau } \frac{ t }{ \tau } \left( \frac{ \varepsilon }{ t^{ n - 1 } \tau^{ 1 - \delta } } \right)^2 d \mu_{ H_\tau } \\
& \lesssim \frac{ \varepsilon^2 }{ \tau^{ 2 - 2 \delta } } \int_0^\infty \frac{ r^{ n - 1 } }{ t^{ 2 n - 2 } } \, d r \\
& \lesssim \frac{ \varepsilon^2 }{ \tau^{ 2 - 2 \delta } t^{ n - 2 } } \\
& \lesssim \frac{ \varepsilon^2 }{ \tau^{ n - 2 \delta } } .
\end{align*}
Similarly, the inhomogeneous part $ F_2 = K F_0 $ is estimated by using Lemma \ref{lem inhomo} as follows:
\begin{align*}
\int_{ H_\tau } \frac{ t }{ \tau } \left( \int_\bbrn | F_2 | \, d v \right)^2 d \mu_{ H_\tau } & \lesssim \int_{ H_\tau } \frac{ t }{ \tau } \left( \int_\bbrn \| K \| | F_0 | \, d v \right)^2 d \mu_{ H_\tau } \\
& \lesssim \int_{ H_\tau } \frac{ t }{ \tau } \left( \int_\bbrn \| K \|^2 | F_0 | \, d v \right) \left( \int_\bbrn | F_0 | \, d v \right) d \mu_{ H_\tau } \\
& \lesssim \int_{ H_\tau } \frac{ t }{ \tau } \hat{ e } ( \| K \|^2 | F_0 | ) \frac{ 1 }{ t^{ n - 1 } \tau } \hat{ E }_n ( F_0 ) \, d \mu_{ H_\tau } \\
& \lesssim \frac{ 1 }{ \tau^n } \hat{ E }_K \hat{ E }_N ( f ) \\
& \lesssim \frac{ \varepsilon^2 }{ \tau^{ n - 2 \delta } } ,
\end{align*}
which completes the proof of the proposition.
\end{proof}

\section*{Acknowledgments}
This work was supported by the National Research Foundation of Korea (NRF) grant funded by the Korea government (MSIT) (No.\ RS-2024-00451692). The author also acknowledges financial support under the project PID2024-155175NB-I00 (Agencia Estatal de Investigaci{\'o}n and Ministerio de Ciencia, Innovaci{\'o}n y Universidades). Part of this work was carried out while the author was visiting the Laboratoire Jacques--Louis Lions. The author thanks Jacques Smulevici for his hospitality and for helpful discussions.

\end{document}